\documentclass{article}       % onecolumn (second format)
\usepackage{graphicx}

\usepackage[english]{babel}
\usepackage[scale=0.75]{geometry}
\usepackage{bbold}
\usepackage{tikz}
\usetikzlibrary{decorations.pathreplacing}
\usetikzlibrary{arrows,decorations.markings}
\usepackage{amsmath}
\usepackage{amssymb}
\usepackage{amsthm}
\usepackage{mathrsfs}
\usepackage{stmaryrd}
\usepackage{amsfonts}
\usepackage{dsfont}
\usepackage{centernot}
\usepackage[colorlinks=true,linkcolor=blue,linktocpage=true]{hyperref}
\hypersetup{
colorlinks=true, %colorise les liens
breaklinks=true, %permet le retour à la ligne dans les liens trop longs
urlcolor= blue, %couleur des hyperliens
linkcolor= red, %couleur des liens internes
citecolor=blue, %couleur des références 
}
\usepackage{hyperref}
\newcommand{\abs}[1]{\left| #1 \right|}
\newcommand{\norme}[1]{\left\| #1 \right\|}
\newcommand{\acc}[1]{\left\{ #1 \right\}}
\newcommand{\croch}[1]{\left[ #1 \right]}

\newcommand{\R}{\mathbb{R}}
\newcommand{\N}{\mathbb{N}}
\newcommand{\Z}{\mathbb{Z}}
\newcommand{\Proba}{\mathbb{P}}
\newcommand{\Probapq}{\phi_{n,\,p,\,q}^1}
\newcommand{\Probapqxi}{\phi_{n,\,p,\,q}^{\xi}}
\newcommand{\Probap}{\phi_{n,\,p,\,2}^1}

\newcommand{\Probapqf}{\phi_{n,\,p,\,q}^0}
\newcommand{\Probapn}{\phi_{n,\,p_n,\,2}^1}

\newcommand{\Ed}{\mathbb{E}^2}
\newcommand{\En}{\mathbb{E}_n}
\newcommand{\Ent}[1]{\left\lfloor #1\right\rfloor}
\newcommand{\Ceil}[1]{\left\lceil #1\right\rceil}
\newcommand{\souspreuve}[1]{\paragraph{#1}}%{\noindent\textbf{#1}}
\newcommand{\suitepreuve}[1]{\paragraph{#1}}%{\vspace{0.2cm}\souspreuve{#1}}
\newcommand{\Czero}{\mathcal{C}_0}
\newcommand{\Cun}{\mathcal{C}_1}
\newcommand{\Clusdec}{\mathcal{C}_n^-}
\newcommand{\restrict}[1]{\raisebox{-.5ex}{$|$}_{#1}} 
\newcommand{\demi}{\frac{1}{2}}
\newcommand{\cvninfty}{\stackrel{n\rightarrow\infty}{\longrightarrow}}
\newcommand{\limn}{\lim\limits_{n\rightarrow\infty}}
\newcommand{\limsupn}{\limsup\limits_{n\rightarrow\infty}}
\newcommand{\liminfn}{\liminf\limits_{n\rightarrow\infty}}
\newcommand{\eqninfty}{\stackrel{n\rightarrow\infty}{\sim}}
\newcommand\numberthis{\addtocounter{equation}{1}\tag{\theequation}}
\newcommand{\diam}{{\rm diam}}
\newcommand{\quadet}{\quad\text{and}\quad}
\newcommand{\qquadet}{\qquad\text{and}\qquad}
\newcommand{\quadou}{\quad\text{where}\quad}
\newcommand{\qquadou}{\qquad\text{where}\qquad}

\newcommand{\quadavec}{\quad\text{with}\quad}
\newcommand{\qquadavec}{\qquad\text{with}\qquad}

\newcommand{\quadimplique}{\quad\Rightarrow\quad}
\newcommand{\qquadimplique}{\qquad\Rightarrow\qquad}

\newcommand{\modulo}[1]{\,\left[\text{mod. }#1\right]}
\newcommand{\connecte}{\stackrel{\omega}{\longleftrightarrow}}
\newcommand{\Edge}[1]{\mathbb{E}\croch{#1}}
\newcommand{\guillemets}[1]{``#1''}
\newcommand{\Tns}{T_n^\star}
\newcommand{\Mn}{\mathcal{M}_n}
\newcommand{\omegaH}{\omega_{H(\omega)}}
\makeatletter
\newcommand*{\biggg}[1]{{\hbox{$\left#1\vbox to20.5\p@{}\right.\n@space$}}}
\newcommand*{\Biggg}[1]{{\hbox{$\left#1\vbox to23.5\p@{}\right.\n@space$}}}
\makeatother
\newtheorem{theorem}{Theorem}
\newtheorem{proposition}{Proposition}
\newtheorem{lemma}{Lemma}

\newtheorem{definition}{Definition}

\newcommand{\correct}[1]{#1}

\begin{document}

\title{A planar Ising model of self-organized criticality}

\author{Nicolas Forien
	\thanks{\'Ecole normale sup\'erieure, CNRS, PSL University, 75005, Paris, France}\ 
	\thanks{Universit\'e Paris-Saclay, CNRS, Laboratoire de math\'ematiques d'Orsay, 91405, Orsay, France}\ 
	\thanks{Aix Marseille Univ, CNRS, Centrale Marseille, I2M, Marseille, France}\ 
	\thanks{I wish to thank Rapha\"el Cerf for suggesting this problem, and for fruitful discussions.}}

\maketitle

\begin{abstract}
We consider the planar Ising model in a finite square box and we replace the temperature parameter with a function depending on the magnetization. This creates a feedback from the spin configuration onto the parameter, which drives the system towards the critical point. Using the finite-size scaling results of~\cite{CerfMessikh2011}, we show that, when the size of the box grows to infinity, the temperature concentrates around the critical temperature of the planar Ising model on the square lattice.
\end{abstract}

\tableofcontents

\section{Introduction}

\subsection{Definition of the model and convergence result}

In this article, we build a simple variant of the two-dimensional Ising model which presents a phenomenon of \guillemets{self-organized criticality}. To define this model, we consider square boxes~$\Lambda(n)\subset\Z^2$ of side~$n$, we choose a real parameter~$a>0$, and we set, for any spin configuration~$\sigma:\Lambda(n)\rightarrow\acc{-,+}$,
$$T_n(\sigma)\ =\ \frac{\big(m(\sigma)\big)^2}{n^{2a}}\ =\ \frac{1}{n^{2a}}\Bigg(\,\sum_{x\in\Lambda(n)}{\sigma(x)}\,\Bigg)^2\,.$$
We then define the probability distribution
$$\mu_n\ :\ \sigma\in\acc{-,+}^{\Lambda(n)}\ \longmapsto\ \frac{1}{Z_n}\mu_{n,\,T_n(\sigma)}^+(\sigma)\,,$$
where~$Z_n$ is the appropriate normalization constant and~$\mu_{n,\,T}^+$ is the standard Ising measure at temperature~$T$ \correct{with}~$+$ boundary conditions on the box~$\Lambda(n)$ \correct{and} no external magnetic field (see section~\ref{section_def_Ising} for the precise definition). In our model, the fixed temperature~$T$ of the Ising model is replaced with this function~$T_n$ of the configuration itself, creating a feedback from the configuration onto the temperature parameter. The goal is to obtain a model whose temperature concentrates around the critical temperature~$T_c$ of the Ising model when the size of the box grows to infinity, without having to tune a parameter to a precise critical value. We prove, in dimension~$2$, the following convergence result:

\begin{theorem}
\label{Ising_thm_CV}
If the parameter~$a$ is chosen such that~$81/41<a<2$, then the law of~$T_n$ under~$\mu_n$ converges to~$\delta_{T_c}$ when~$n\rightarrow\infty$, and we have the following estimate on the convergence speed:
\begin{equation}
\label{estimee_vitesse}
\forall\varepsilon>0\qquad
\limsupn\,\frac{1}{n}\,\ln\mu_n\Big(\,\abs{T_n-T_c}\geqslant\varepsilon\,\Big)\ <\ 0\,.
\end{equation}
\end{theorem}

We do not think that this constant~$81/41$ is optimal, since it comes from the hypotheses of~\cite{CerfMessikh2011}, which are not deemed to be optimal. We explain in paragraph~\ref{section_dependance_messikh} how the exponent~$81/41$ would evolve if the results of~\cite{CerfMessikh2011} were to be improved.

\subsection{Self-organized criticality}

Most lattice models which present a phase transition have a particularly interesting behaviour at their critical point, with (conjectured or proven) properties of conformal invariance, self-similarity and power-law correlations, which turn out to be common within a large class of similar models in statistical mechanics. The physicists Per Bak, Chao Tang and Kurt Wiesenfeld noted in~\cite{BTW} that these universal features can be observed in various physical or biological systems. But, these properties being very specific to the critical point in a phase transition, it seems surprising to meet them by chance in nature, since it should be very unlikely to have the parameters of a physical system precisely tuned to their critical value.

According to~\cite{BTW}, an explanation is that some physical systems tend to be naturally attracted towards a critical state: this phenomenon is called~\guillemets{self-organized criticality}. Examples of self-critical models include the sandpile model~\cite{JAR}, forest fires~\cite{ForestFires}, avalanche processes~\cite{BF}, biological evolution~\cite{BDF}, neural systems~\cite{Neural}, or sociology~\cite{SocialPerco}, and self-organized criticality was recently experimentally observed in an ultra-cold atomic gas~\cite{SOCNature}. But these systems turn out to be very difficult to analyze rigorously and there are few self-critical models which are simple enough to be amenable to mathematical study but complex enough to enclose the relevant features of self-organized criticality.

A natural idea to obtain a self-critical model, described in~\cite{Sornette}, is to start with a model which exhibits a phase transition, and to introduce a feedback by replacing the control parameter (e.g., temperature) with a function of the configuration. If this feedback function is well chosen, it can drive the system towards its critical point. In~\cite{CG}, a simple model of self-organized criticality was built from the generalized Curie-Weiss Ising model by using this technique. In~\cite{ModelsofSOCinPerco}, we have defined a similar model constructed from Bernoulli percolation. Therefore, it was natural to try to extend this result to the Ising model, using the random-cluster representation. But the general technique of~\cite{ModelsofSOCinPerco} turned out to be hard to apply in the more general setting of FK-percolation, because of the non-poissonian way the edges become open when the percolation parameter is raised to the critical point from below~\cite{NearCriticalFKIsing}.

Therefore, we adopt a slightly different approach, and we use the results of~\cite{CerfMessikh2011} about the near-critical regime of the planar FK-Ising model. Hence, this proof is very specific to the two-dimensional setting, and we are only able to study a very small window for the parameter~$a$. An improvement of the hypotheses of~\cite{CerfMessikh2011} would enlarge our window, but since it only deals with the slightly supercritical regime (when~$p-p_c$ tends to zero but is positive), a different method would be needed to study the case of lower parameters~$a$.

\correct{Although we have only studied our model in the planar case, we expect that the result could be generalized in other dimensions, but this would require other ingredients. In particular, the case of high dimensions could probably be investigated using lace expansion techniques~\cite{HS90,Sakai07}.}

\subsection{Outline of the paper}

To control the deviations of the random variable~$T_n$, we write, for~$\varepsilon>0$,
\begin{align*}
\mu_n\Big(\,T_n\geqslant T_c+\varepsilon\,\Big)
\ &=\ \frac{1}{Z_n}\sum_{\sigma\in\acc{-,+}^{\Lambda(n)}}{\mathbb{1}\big[\,T_n(\sigma)\geqslant T_c+\varepsilon\,\big]\mu_{n,\,T_n(\sigma)}^+(\sigma)}\\
\ &=\ \frac{1}{Z_n}\sum_{b=0}^{n^2}{\mathbb{1}\croch{\,b^2/n^{2a}\geqslant T_c+\varepsilon\,}\sum_{\sigma\;:\;\abs{m(\sigma)}=b}{\mu_{n,\,b^2/n^{2a}}^+(\sigma)}}\\
\ &=\ \frac{1}{Z_n}\sum_{b=0}^{n^2}{\mathbb{1}\croch{\,b\geqslant n^{a}\sqrt{T_c+\varepsilon}\,}\mu_{n,\,b^2/n^{2a}}^+\Big(\,\abs{m}=b\,\Big)}\\
\ &\leqslant\ \frac{n^2+1}{Z_n}\sup_{T\geqslant T_c+\varepsilon}\mu_{n,\,T}^+\Big(\,\abs{m}\geqslant n^a\sqrt{T_c+\varepsilon}\,\Big)\,.\numberthis\label{outline_majo_surcritique}
\end{align*}
Similarly, we have, for~$\varepsilon>0$,
\begin{equation}
\label{outline_majo_souscritique}
\mu_n\Big(\,T_n\leqslant T_c-\varepsilon\,\Big)
\ \leqslant\ \frac{n^2+1}{Z_n}\sup_{T\leqslant T_c-\varepsilon}\mu_{n,\,T}^+\Big(\,\abs{m}\leqslant n^a\sqrt{T_c-\varepsilon}\,\Big)\,.
\end{equation}
Therefore, our strategy \correct{consists of} proving exponential decay results for
\begin{equation}
\label{sups_a_majorer}
\sup_{T\geqslant T_c+\varepsilon}\mu_{n,\,T}^+\Big(\,\abs{m}\geqslant A n^a\,\Big)
\qquadet
\sup_{T\leqslant T_c-\varepsilon}\mu_{n,\,T}^+\Big(\,\abs{m}\leqslant An^a\,\Big)
\end{equation}
for fixed~$A>0$ and~$\varepsilon>0$, when the size~$n$ of the box tends to infinity. This is done in sections~\ref{bigsec_surcritique} and~\ref{bigsec_souscritique}, after some general definitions and notations are introduced in section~\ref{bigsec_defs}. These exponential estimates are quite standard for a fixed temperature~$T\neq T_c$, but since we do not have a \correct{monotonicity} property with respect to~$T$ for the magnetization, uniform estimates on the intervals~$[0,\,T_c-\varepsilon]$ and~$[T_c+\varepsilon,\,+\infty)$ cannot be deduced from the pointwise exponential decay. Therefore, we work with the random-cluster representation of the Ising model, whose \correct{monotonicity} property helps us to obtain uniform exponential bounds for~(\ref{sups_a_majorer}).

But these exponential decay results are not enough to prove theorem~\ref{Ising_thm_CV}, since it could be the case that the denominator~$Z_n$ in~(\ref{outline_majo_surcritique}) and~(\ref{outline_majo_souscritique}) also decays exponentially. Thus, we need to show at least that~$Z_n$ does not decay as fast as the two quantities in~(\ref{sups_a_majorer}). This is the key point of our proof, detailed in section~\ref{bigsec_minoZn}. In paragraph~\ref{section_heuristique_Zn} below, we present the strategy to obtain this lower bound on the partition function, which is close to the strategy followed in~\cite{ModelsofSOCinPerco}. But here, instead of building a monotone coupling of configurations and looking for a fixed point, we guess the value of this fixed point \correct{using} the results of finite-size scaling in~\cite{CerfMessikh2011}. \correct{These results, which are the crucial ingredient in our proof,} indicate which speed of convergence to~$p_c$ is required to obtain a given magnetization. \correct{But the estimates of~\cite{CerfMessikh2011} are very specific to the two-dimensional case, since they rely on the exact computations of Onsager~\cite{Onsager}.} Therefore, our method is very specific to the two-dimensional Ising model, and an extension to higher dimensions would require other ingredients.

In all three regimes (supercritical in section~\ref{bigsec_surcritique}, subcritical in section~\ref{bigsec_souscritique}, and near-critical in section~\ref{bigsec_minoZn}), we control the magnetization of the Ising model with the help of the random-cluster model, which is linked to the Ising model through the Edwards-Sokal coupling (see paragraph~\ref{thm_Edwards_Sokal}). In an Ising configuration obtained from a FK-percolation configuration with this coupling, the magnetization is the result of two factors: on the one hand, the number of vertices connected to the boundary of the box, and on the other hand, the fluctuations coming from the spins attributed to the clusters which do not touch the boundary of the box. Therefore, we have to monitor both factors in the random-cluster model to obtain a control on the magnetization in the related Ising model.

To control the dependence of our result on the hypotheses of~\cite{CerfMessikh2011}, we prove our lower bound on~$Z_n$ in section~\ref{bigsec_minoZn} under some finite-size scaling assumptions, which are proved in~\cite{CerfMessikh2011} but with hypotheses which are not deemed optimal. We show in section~\ref{bigsec_FSS} how these scaling assumptions follow from~\cite{CerfMessikh2011}. This allows us to discuss how our admissible range for the parameter~$a$ would improve if the postulates were extended to a broader window of percolation parameters.

\subsection{Improvement of our exponent under finite-size scaling assumptions}

\label{section_dependance_messikh}

Our proof relies on the results of~\cite{CerfMessikh2011}, which give information on the set of the sites connected to the boundary in the planar FK-Ising model in a joint regime where~$p\rightarrow p_c$ and~$n\rightarrow\infty$ simultaneously, with~\smash{$p-p_c\gg n^{-8/41}$}.
This value of~$8/41$ is not deemed optimal, since it is believed that the FK-percolation model should have a supercritical behaviour as long as~$n$ is much larger than the correlation length, which scales like~$(p-p_c)^{-1}$~\cite{NearCriticalFKIsing}.
As we will see in paragraph~\ref{section_heuristique_Zn}, we will study FK-percolation in the regime~$p-p_c\ \sim\ n^{8a-16}$, hence our assumption~$a>81/41$, in order to have~$8a-16>-8/41$. To keep track of the influence of this condition on our exponent~$a$, we quote below which of the results of~\cite{CerfMessikh2011} are needed for our proof.

\begin{definition}
We say that an exponent~$s>0$ satisfies the finite-size scaling assumptions, which we will denote by~$\mathcal{FSS}(s)$, if for all~$K,\,\delta>0$ and for any real sequence~$p_n\in[0,1]$, we have
$$p_n-p_c(2)\ \eqninfty\ \frac{K}{n^s}
\correct{\qquadimplique}
\limn\Probapn\big(\mathcal{F}_n\big)\ =\ 1\,,$$
$$\correct{\text{with}\qquad}
\mathcal{F}_n\ =\ \acc{\,\begin{array}{c}
\abs{\Mn}\leqslant(1+\delta)\theta(p_n)\abs{\Lambda(n)},\\
\abs{\Mn\cap\Lambda(n_1)}\geqslant(1-\delta)\theta(p_n)\abs{\Lambda(n_1)}\\
\max\big\{\abs{C(x)}\,:\,x\centernot\longleftrightarrow\partial\Lambda(n)\big\}\leqslant n^{s+1/2}
\end{array}\,}\,,$$
where~$n_1=\Ent{5n/6}$,~$\phi_{n,\,p,\,q}^1$ is the finite-volume random-cluster measure with wired boundary conditions on the box~$\Lambda(n)$ and~$\Mn$ denotes the set of vertices connected to the boundary~$\partial\Lambda(n)$ of the box.
\end{definition}

We translate the results of~\cite{CerfMessikh2011} into the following proposition:

\begin{proposition}
\label{prop_messikh}
We have~$\mathcal{FSS}(s)$ for all~$s<8/41$.
\end{proposition}

The next theorem shows how the constant~$81/41$ would evolve if the results of~\cite{CerfMessikh2011} were to be improved. Together with proposition~\ref{prop_messikh}, it easily implies theorem~\ref{Ising_thm_CV}. We do not know whether the value~$31/16$ is optimal.

\begin{theorem}
\label{Ising_thm_CV_avec_dep_Messikh}
If~$a\in(31/16,\,2)$ is such that~$\mathcal{FSS}(16-8a)$ holds, then the law of~$T_n$ under~$\mu_n$ converges to~$\delta_{T_c}$ when~$n\to\infty$, and the estimate~(\ref{estimee_vitesse}) holds.
\end{theorem}

\subsection{Heuristics for the lower bound on the partition function}

\label{section_heuristique_Zn}

We explain here the strategy to obtain a lower bound on~$Z_n$. This is the key step of our proof, detailed in section~\ref{bigsec_minoZn}. We take~$a\in(31/16,\,2)$ such that the assumptions~$\mathcal{FSS}(16-8a)$ hold, and we start by rewriting~$Z_n$ as
\begin{equation}
\label{reecriture_Zn}
Z_n\ =\ \sum_{\sigma\in\acc{-,+}^{\Lambda(n)}}{\mu_{n,\,T_n(\sigma)}^+(\sigma)}
\ =\ \sum_{b=-n^2}^{n^2}{\mu_{n,\,b^2/n^{2a}}^+\Big(\,m=b\,\Big)}\,.
\end{equation}
\correct{As can be seen in~(\ref{reecriture_Zn}), our partition function~$Z_n$ is in a way related to the \guillemets{steepness} of the phase transition in the Ising model.
If the phase transition were too brutal, then~$Z_n$ would be very small, since it would be very unlikely to observe a value of~$b$ with an intermediate magnetization, of order~$n^a$.

Since we bypass this study of the steepness of the phase transition by using results following from the exact computations of Onsager, one may wonder whether, in return, information about the phase transition of the Ising model could be deduced from our results.
This remains an open question, but it seems unlikely to us that our lower bound on~$Z_n$ could be useful as it is, because we do not think it to be optimal.
Nevertheless, the cutting and coloring technique that we develop may be useful for future applications in the study of the near-critical regime of the Ising or Potts models.}

\suitepreuve{Quest for the fixed point~$b_n$:}
To obtain our lower bound, we search for a value of~$b_n$ such that, at~$T=b_n^2/n^{2a}$, the magnetization is exactly~$b_n$ with probability high enough. If we choose~$b_n$ such that~$b_n^2/n^{2a}\centernot{\longrightarrow}T_c$, then we cannot obtain a slower decay than the decays proved in the subcritical regime (in section~\ref{bigsec_surcritique}) and in the supercritical regime (in section~\ref{bigsec_souscritique}). Thus, we will choose~$b_n$ such that~$b_n\sim n^a\sqrt{T_c}$. Thanks to the Edwards-Sokal coupling (see paragraph~\ref{thm_Edwards_Sokal}), the Ising model at temperature~$T=b_n^2/n^{2a}$ can be recovered from the random-cluster model with parameters~$q=2$ and~$p=\varphi_n(b_n)$, where~$\varphi_n$ is the function defined by
\begin{equation}
\label{definition_varphi}
\varphi_n\ :\ b\in\big\{\,-n^2,\,\ldots,\,n^2\,\big\}\ \longmapsto\ \left\{\begin{aligned}
&1\text{ if }b=0\,,\phantom{\left(\demi\right)}\\
&1-\exp\left(-\frac{2n^{2a}}{b^2}\right)\text{ otherwise.}
\end{aligned}\right.
\end{equation}
For any percolation configuration~$\omega:\En\rightarrow\acc{0,1}$, the set of vertices connected to the boundary is denoted by~$\Mn(\omega)$ (see section~\ref{bigsec_defs} for all these general definitions). The idea is to choose~$b_n$ such that, under the law~\smash{$\phi_{n,\,\varphi_n(b_n),\,2}^1$}, the number of vertices connected to the boundary is typically of the order of~$b_n$, and then to control the magnetization of the clusters which do not touch the boundary of the box. Therefore, a natural strategy \correct{consists of} proving a lower bound on the probability that, on the one hand,~$\abs{\Mn(\omega)}=b_n$, and on the other hand, the contribution of the other clusters to the magnetization cancels out, so as to attain a magnetization exactly equal to~$b_n$. To obtain a percolation configuration~$\omega$ such that~$\abs{\Mn(\omega)}=b_n$, it is simpler to require as a first step that~$\abs{\Mn}$ be between~$b_n$ and~$\lambda n^a$, with~$\lambda>\sqrt{T_c}$, and in a second step to close some edges to reach exactly~$b_n$. Therefore, we take~$b_n$ satisfying
$$\theta\big(\varphi_n(b_n)\big)\abs{\Lambda(n)}\ \eqninfty\ \mu n^a\,,
\qquadavec
\sqrt{T_c}\ <\ \mu\ <\ \lambda\,.$$
According to the asymptotics for~$\theta(p)$ given by the exact computations of Onsager~\cite{Onsager} and Yang~\cite{Yang}, we need that
$$\varphi_n(b_n)-p_c(2)\ \eqninfty\ \frac{p_c(2)}{8}\left(\frac{\mu}{n^{2-a}}\right)^8\,.$$
We will then have, under the assumptions~$\mathcal{FSS}(16-8a)$,
\begin{equation}
\label{encadrement_lambda}
\limn\phi_{n,\,\varphi_n(b_n),\,2}^1\Big(\,b_n\,\leqslant\,\abs{\Mn}\,\leqslant\,\lambda n^a\,\Big)
\ =\ 1\,.
\end{equation}
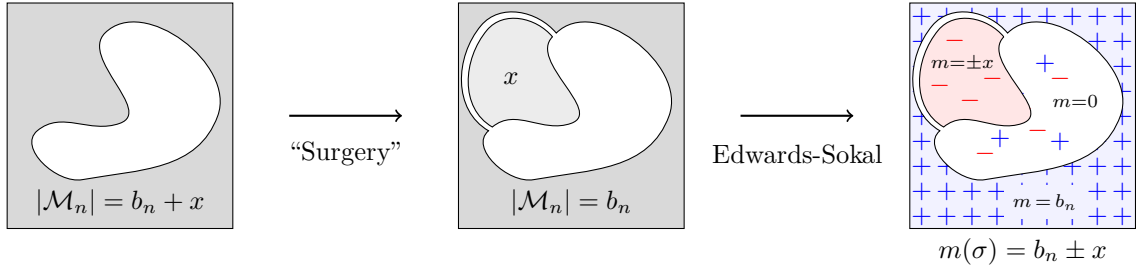
\begin{figure}[ht]
\begin{center}
\begin{tikzpicture}
\filldraw[fill=gray!30] (-5,0) --++ (3,0) --++ (0,3) --++ (-3,0) -- cycle;
\filldraw[fill=white] (-4,0.65) to[out=10, in=-110] (-2.2,1.5) to[out=70, in=45] (-3.7,2.6) to[out=-135, in=20] (-3.5,1.4) to[out=-160, in=40] (-4.6,1.3) to[out=-140, in=-170] (-4,0.65);
\draw (-3.5,0.35) node{$\abs{\Mn}=b_n+x$};

\draw[->,thick] (-1.25,1.5) --++ (1.5,0);
\draw (-0.5,1) node{\guillemets{Surgery}};

\filldraw[fill=gray!30] (1,0) --++ (3,0) --++ (0,3) --++ (-3,0) -- cycle;
\filldraw[fill=white] (2,0.65) to[out=10, in=-110] (3.8,1.5) to[out=70, in=45] (2.3,2.6) to[out=130, in=50] (1.3,2.7) to[out=-130, in=150] (1.4,1.3) to[out=-140, in=-170] (2,0.65);
\filldraw[fill=gray!15] (2.25,2.55) to[out=-135, , in=20] (2.5,1.4) to[out=-160, in=40] (1.45,1.35) to[out=150, in=-130] (1.38,2.64) to[out=50, in=130] (2.25,2.55);
\draw (2.5,0.35) node{$\abs{\Mn}=b_n$};
\draw (1.7,2) node{$x$};

\draw[->,thick] (4.75,1.5) --++ (1.5,0);
\draw (5.5,1) node{Edwards-Sokal};

\filldraw[fill=blue!5] (7,0) --++ (3,0) --++ (0,3) --++ (-3,0) -- cycle;
\foreach \x in {0,...,8}
{\foreach \y in {0,...,8}
{\draw[blue] (\x/3+7.16,\y/3+0.16) node{$+$};}
}
\filldraw[fill=white] (8,0.65) to[out=10, in=-110] (9.8,1.5) to[out=70, in=45] (8.3,2.6) to[out=130, in=50] (7.3,2.7) to[out=-130, in=150] (7.4,1.3) to[out=-140, in=-170] (8,0.65);
\draw[red] (8,1) node{$-$};
\draw[red] (9,2) node{$-$};
\draw[red] (8.5,1.6) node{$-$};
\draw[red] (8.7,1.3) node{$-$};
\draw[red] (7.9,1.8) node{$-$};
\draw[blue] (8.2,1.2) node{$+$};
\draw[blue] (8.4,1.8) node{$+$};
\draw[blue] (8.8,2.2) node{$+$};
\draw[blue] (9,1.16) node{$+$};
\filldraw[fill=red!10] (8.25,2.55) to[out=-135, , in=20] (8.5,1.4) to[out=-160, in=40] (7.45,1.35) to[out=150, in=-130] (7.38,2.64) to[out=50, in=130] (8.25,2.55);
\draw[red] (7.6,2.5) node{$-$};
\draw[red] (8.1,2) node{$-$};
\draw[red] (7.4,1.9) node{$-$};
\draw[red] (7.8,1.7) node{$-$};
\draw (8.8,0.35) node[fill=blue!5]{${\scriptstyle m\,=\,b_n}$};
\draw (9.2,1.7) node{${\scriptstyle m=0}$};
\draw (7.7,2.2) node{${\scriptstyle m=\pm x}$};
\draw (8.5,-0.3) node{$m(\sigma)=b_n\pm x$};
\end{tikzpicture}
\end{center}
\caption{A strategy which seems natural is, when there are~$b_n+x$ vertices connected to the boundary, to disconnect a piece of size~$x$ to obtain~$\abs{\Mn}=b_n$. But, doing so, the disconnected piece makes it harder to bound the magnetization in the corresponding Ising model, and hence we do not proceed this way.}
\end{figure}

\vspace{-\baselineskip}
\suitepreuve{Trick of the halfway cut:}
Starting from a configuration satisfying~(\ref{encadrement_lambda}), a natural idea is to close a certain set of edges to disconnect exactly~$\abs{\Mn}-b_n$ vertices from the boundary, in order to obtain~$\abs{\Mn}=b_n$.
But a problem arises when trying to control the fluctuations of the magnetization of the clusters which do not touch the boundary, because the set we have disconnected from the boundary has a size of order~$n^a$. Hence, even if this set is not necessarily connected, its contribution to the magnetization might be too high to be compensated. Therefore, we bypass this problem by disconnecting twice less vertices than needed to reach~$\abs{\Mn}=b_n$, and then forcing the set of the vertices thus disconnected to choose a negative spin. In that way, the overall magnetization resulting from the vertices connected to the boundary and from this set of disconnected vertices altogether will be~$b_n$, and the only thing left to do will be to force the magnetization of the other clusters to cancel out.
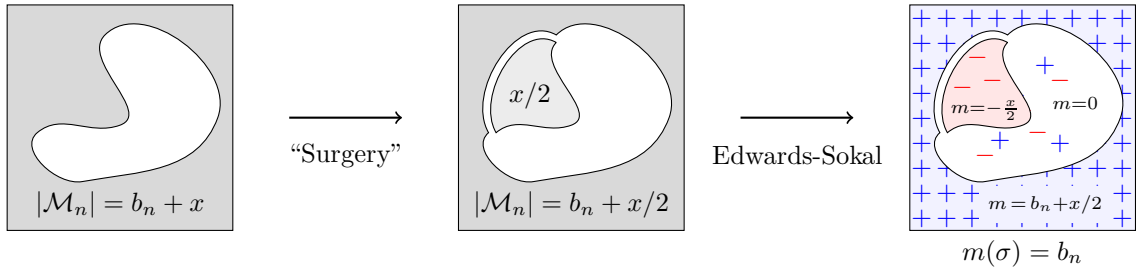
\begin{figure}[ht]
\begin{center}
\begin{tikzpicture}
\filldraw[fill=gray!30] (-5,0) --++ (3,0) --++ (0,3) --++ (-3,0) -- cycle;
\filldraw[fill=white] (-4,0.65) to[out=10, in=-110] (-2.2,1.5) to[out=70, in=45] (-3.7,2.6) to[out=-135, in=20] (-3.5,1.4) to[out=-160, in=40] (-4.6,1.3) to[out=-140, in=-170] (-4,0.65);
\draw (-3.5,0.35) node{$\abs{\Mn}=b_n+x$};

\draw[->,thick] (-1.25,1.5) --++ (1.5,0);
\draw (-0.5,1) node{\guillemets{Surgery}};

\filldraw[fill=gray!30] (1,0) --++ (3,0) --++ (0,3) --++ (-3,0) -- cycle;
\filldraw[fill=white] (2,0.65) to[out=10, in=-110] (3.8,1.5) to[out=70, in=45] (2.3,2.6) to[out=150, in=120] (1.4,1.3) to[out=-140, in=-170] (2,0.65);
\filldraw[fill=gray!15] (2.3,2.5) to[out=-135, , in=20] (2.5,1.4) to[out=-160, in=40] (1.5,1.3) to[out=120, in=150] (2.3,2.5);
\draw (2.5,0.35) node{$\abs{\Mn}=b_n+x/2$};
\draw (1.95,1.8) node{$x/2$};

\draw[->,thick] (4.75,1.5) --++ (1.5,0);
\draw (5.5,1) node{Edwards-Sokal};

\filldraw[fill=blue!5] (7,0) --++ (3,0) --++ (0,3) --++ (-3,0) -- cycle;
\foreach \x in {0,...,8}
{\foreach \y in {0,...,8}
{\draw[blue] (\x/3+7.16,\y/3+0.16) node{$+$};}
}
\filldraw[fill=white] (8,0.65) to[out=10, in=-110] (9.8,1.5) to[out=70, in=45] (8.3,2.6) to[out=150, in=120] (7.4,1.3) to[out=-140, in=-170] (8,0.65);
\draw[red] (8,1) node{$-$};
\draw[red] (9,2) node{$-$};
\draw[red] (8.5,1.6) node{$-$};
\draw[red] (8.7,1.3) node{$-$};
\draw[red] (7.9,1.8) node{$-$};
\draw[blue] (8.2,1.2) node{$+$};
\draw[blue] (8.4,1.8) node{$+$};
\draw[blue] (8.8,2.2) node{$+$};
\draw[blue] (9,1.16) node{$+$};
\filldraw[fill=red!10] (8.3,2.5) to[out=-135, , in=20] (8.5,1.4) to[out=-160, in=40] (7.5,1.3) to[out=120, in=150] (8.3,2.5);
\draw[red] (7.9,2.3) node{$-$};
\draw[red] (8.1,2) node{$-$};
\draw[red] (7.7,1.9) node{$-$};
\draw (8.8,0.35) node[fill=blue!5]{${\scriptstyle m\,=\,b_n+x/2}$};
\draw (9.2,1.7) node{${\scriptstyle m=0}$};
\draw (8,1.6) node{${\scriptstyle m=-\frac{x}{2}}$};
\draw (8.5,-0.3) node{$m(\sigma)=b_n$};
\end{tikzpicture}
\end{center}
\caption{Our strategy to prove the lower bound on~$Z_n$, with the surgery step implemented in section~\ref{section_chirurgie} and the forced colouring of section~\ref{section_minoration_Zn}.}
\end{figure}

\suitepreuve{Construction of the cutting:}
How can we find a set of edges whose closure disconnects from the boundary exactly~\smash{$\Ceil{(\abs{\Mn}-b_n)/2}$}
vertices, and which is not too big, for the closure of these edges not to be too \guillemets{expensive}? The idea is to consider a sub-box of side~$n_1=\Ent{5n/6}$, and to use again the hypotheses~$\mathcal{FSS}(16-8a)$, which ensure that
$$\limn\phi_{n,\,\varphi_n(b_n),\,2}^1\Big(\,\abs{\Mn\cap\Lambda\left(n_1\right)}\,\geqslant\,\nu n^a\,\Big)\ =\ 1\,,
\quadavec
\nu\ <\ \left(\frac{5}{6}\right)^2\mu\,.$$
If also~$\nu>(\lambda-\sqrt{T_c})/2$,
then we have, with probability converging to one,
$$\abs{\Mn\,\backslash\,\Lambda\left(n_1\right)}
\ \leqslant\ \frac{\abs{\Mn}}{2}+\frac{\lambda n^a}{2}-\nu n^a
\ \leqslant\ \frac{\abs{\Mn}}{2}+\frac{\sqrt{T_c}}{2}n^a
\ \simeq\ \frac{\abs{\Mn}+b_n}{2}\,,$$
meaning that it is enough to disconnect the smaller box~$\Lambda(n_1)$ in order to go half the way separating~$\abs{\Mn}$ from~$b_n$.
In what follows, we will take as parameters~$\lambda=4$,~$\mu=3$ and~$\nu=2$. Using the pigeonhole principle, we will show that it does not cost more than~$O\left(n^{a-1}\right)$ edges to disconnect this sub-box from the boundary (this will be done in lemma~\ref{lemme_proba_Rn}). But doing this, we may have cut too much, so we use the geometrical lemma of~\cite{ModelsofSOCinPerco} to adjust the cutting to obtain the desired result, for a cost of~$O(n^{a/2})$ edges.

\suitepreuve{Control of the fluctuations of the magnetization:}
We then construct an Ising configuration from the percolation configuration we have obtained with the cutting procedure. It does not cost more than~\smash{$2^{-O(n^{a/2})}$} to force the set of the vertices we have disconnected to choose a minus spin, because this set is made up of at most~$O(n^{a/2})$ clusters. Regarding the other clusters which were already not connected to the boundary of the box before the cutting operation, we have a bound~$N$ on their size given by the assumptions~$\mathcal{FSS}(16-8a)$. This allows us to show that, with sufficient probability, the magnetization resulting from these clusters will fall inside the range~$[-N,\,N]$. To compensate this contribution, we will force the spins of the clusters of size~$1$, after having shown that, with high probability, there are are at least~$N$ such unit clusters. All this will allow us to build an Ising configuration with a magnetization exactly equal to~$b_n$, with a small parity issue that we will dodge by choosing~$b_n$ with the right parity, the magnetization being constrained to have the same parity as~$\abs{\Lambda(n)}$.

\subsection{Naturalness of the model}

The probability density of our model is
\begin{equation}
\label{mun_ZZ}
\mu_n\ :\ \sigma\ \longmapsto\ 
\frac{1}{Z_n}\frac{1}{Z_{n,\,T_n(\sigma)}^+}\exp\left(-\frac{\mathcal{H}_n^+(\sigma)}{T_n(\sigma)}\right)\,,
\end{equation}
where~$\mathcal{H}_n^+$ is the Hamiltonian of the standard Ising model, defined in~(\ref{def_Hamiltonien_Ising}).
An other distribution which may look more natural to consider is
\begin{equation}
\label{natural_model}
\mu'_n\ :\ \sigma\ \longmapsto\ 
\frac{1}{Z'_n}\exp\left(-\frac{\mathcal{H}_n^+(\sigma)}{T_n(\sigma)}\right)\,,
\end{equation}
where~$Z'_n$ is the required normalization constant. Most surprisingly, simulations seem to indicate that this other model might be too simple to exhibit the self-critical behaviour we are looking for. We give here a simple heuristics to understand the difference between these two models. When simulating the Ising model with the Glauber dynamics (see section~8.2 of~\cite{GrimmettFK}), the configuration is updated one spin after another, by looking at the impact of the spin flip on the energy. In the standard Ising model, the \guillemets{energy} function writes~$E_{Ising}=\mathcal{H}_n^+/T$.
Therefore, when flipping one spin, the variation of energy is~$\Delta E_{Ising}=\Delta \mathcal{H}_n^+/T$,
and the spin flip is more likely if it leads to a lower value for the Hamiltonian~$\mathcal{H}_n^+$. However, in the model~(\ref{natural_model}), due to the fact that the temperature is no longer a constant, a spin flip leads to a change in energy given by
\begin{equation}
\label{deltaE}
\Delta E_{\mu'_n}\ =\ \Delta\left(\frac{\mathcal{H}_n^+}{T_n}\right)
\ \approx\ \frac{\Delta \mathcal{H}_n^+}{T_n}-\frac{\mathcal{H}_n^+\Delta T_n}{T_n^2}
\ =\ \frac{\mathcal{H}_n^+}{T_n}\left(\frac{\Delta\mathcal{H}_n^+}{\mathcal{H}_n^+}-\frac{\Delta T_n}{T_n}\right)\,.
\end{equation}
There is a competition between the two terms in~(\ref{deltaE}) to influence the spin flip. Therefore, the dynamics of the Hamiltonian is perturbed by the dynamics on the temperature. For our feedback temperature function to create a self-organized behaviour, we need the dynamics on the Hamiltonian to outweigh the drift force in temperature. If this is the case, then the configuration has the time to reach a typical configuration at fixed temperature, and only once this equilibrium is reached, on a longer time scale, the temperature evolves and slowly drives the system towards criticality. Computer simulations confirm this idea that the dynamics on the Hamiltonian and on temperature must not compete but should occur on different time scales, in order to reach a self-critical state. Indeed, if one simulates the model~(\ref{natural_model}) with the Glauber dynamics but without taking into account the change in temperature when deciding the spin flip, one gets a much more promising output, which is improved if the temperature parameter is updated only once in a while, leaving the configuration some time to reach equilibrium at fixed temperature before evaluating a new temperature parameter.

One may wonder why there is not the same unfortunate competition phenomenon with our model~$\mu_n$ given by equation~(\ref{mun_ZZ}). In this model, if we take into account the influence of the partition function in the denominator, then the energy function writes
$$E_{\mu_n}\ =\ \ln Z_{n,\,T_n}^+ + \frac{\mathcal{H}_n^+}{T_n}\,.$$
Therefore, a \correct{spin flip} leads to a change in energy given by
\begin{align*}
\Delta E_{\mu_n}
\ =\ \Delta\left(\ln Z_{n,\,T_n}^++\frac{\mathcal{H}_n^+}{T_n}\right)
\ \approx\ &\frac{\Delta\mathcal{H}_n^+}{T_n}+\frac{\partial\ln Z_{n,\,T}^+}{\partial T}\Delta T_n-\frac{\mathcal{H}_n^+\Delta T_n}{T_n^2}\\
\ \correct{\approx}\ &\frac{\Delta\mathcal{H}_n^+}{T_n}+\left(\mu_{n,\,T_n}^+\big(\mathcal{H}_n^+\big)-\mathcal{H}_n^+\right)\frac{\Delta T_n}{T_n^2}\,.
\end{align*}
Hence, the factor in front of~$\Delta T_n$ is tailored to be smaller than in~(\ref{deltaE}), thanks to the compensation coming from the term~\smash{$\mu_{n,\,T_n}^+\big(\mathcal{H}_n^+\big)$}. This \correct{may} explain why, in our model, the effect of the Hamiltonian overcomes the temperature effect, which ensures that our self-tuning of the temperature parameter is delicate enough to preserve the equilibrium properties of the model.

\correct{Since it involves a global interaction through a feedback from the whole configuration to the temperature parameter, our model may seem somewhat unrealistic from a physical point of view.
But, even in models where the interaction mechanism is only local, it seems that self-organized criticality appears as the result of large-scale instabilities which are created dynamically.
For example, in the abelian sandpile model~\cite{Dhar99ASM}, long-range interactions are hidden behind the fact that the addition of one sand grain can trigger a large avalanche.
Therefore, integrated over a long time, these interactions lead to a global interdependence, with long-range correlations.

In fact, to obtain self-organized criticality, it seems that one cannot avoid either having this sort of dynamics, or defining a global static interaction, like the one in our model.
Thus, our model may be considered not very natural, but we think that it would be hard to define a much more natural static model presenting the phenomenon of self-organized criticality.}

\section{Definitions and notations}

\label{bigsec_defs}

\subsection{Edges and boxes}

The entire article takes place in~$d=2$. The square box of side~$n$ centered at~$0$ is denoted~$\Lambda(n)=[-n/2,\,n/2[^2\cap\Z^2$.
We say that two points~$x,\,y\in\Z^2$ are neighbours if~$\norme{x-y}_1=1$, which is denoted~$x\sim y$. For any~$V\subset\Z^2$, we write
$$\Edge{V}\ =\ \Big\{\acc{x,\,y}\subset V\ :\ x\sim y\,\Big\}\,.$$
We define in this way~$\Ed=\Edge{\Z^2}$ \correct{and}~$\En=\Edge{\Lambda(n)}$.
\correct{For~$V\subset\Z^2$, we define the interior and exterior boundary of~$V$ to be, respectively:}
\begin{align*}
\partial V\ &=\ \Big\{\,x\in V\ :\ \exists\,y\in\Z^2\backslash V\quad x\sim y\,\Big\}\,,\\
\partial^e V\ &=\ \Big\{\,\acc{x,\,y}\in\Ed\ :\ x\in V\quadet y\in\Z^2\backslash V\,\Big\}\,.
\end{align*}
\correct{For~$V\subset\Z^2$ finite and non-empty, we write}~\smash{$\diam\,V=\max\limits_{x,y\in V}\norme{x-y}_\infty$}.

\subsection{Percolation configurations}

An element~$\omega:\En\rightarrow\acc{0,1}$ is called a percolation configuration. Edges~$e\in\En$ such that~$\omega(e)=1$ are \correct{said to be open} in~$\omega$, while the other edges are \correct{said to be closed} in~$\omega$.
The space of configurations is endowed with a natural partial order defined by~$\omega_1\leqslant\omega_2$ if~$\omega_1(e)\leqslant\omega_2(e)$ for all edges~$e\in\En$. If~$\mu_1$ and~$\mu_2$ are two probability measures on~\smash{$\acc{0,1}^{\En}$}, \correct{we write}~$\mu_1\preceq\mu_2$ if~$\mu_1(X)\leqslant\mu_2(X)$ for every increasing~\smash{$X:\acc{0,1}^{\En}\rightarrow\R$}.
\correct{For~$x,\,y\in\Lambda(n)$, we write~\smash{$x \connecte y$}
when there exists an open path from~$x$ to~$y$ in~$\omega$, and if~$Y\subset\Lambda(n)$, we write~$x \connecte Y$ if there exists~$y\in Y$ such that~$x \connecte y$.
If~$x\in\Lambda(n)$, the cluster of~$x$ in~$\omega$ is}
$$C(x)\ =\ C(x,\,\omega)\ =\ \acc{\,y\in\Lambda(n)\ :\ x \connecte y\,}\,.$$
For any set of edges~$H\subset\En$, we define the configuration~$\omega_H$ obtained from~$\omega$ by closing all the edges of~$H$:
$$
\omega_H\ :\ e\in\En\ \longmapsto\ \left\{
\begin{aligned}
&0\text{ if }e\in H\,,\\
&\omega(e)\text{ otherwise.}
\end{aligned}
\right.
$$
The set of the vertices connected to the boundary in~$\omega$ is written
$$\Mn(\omega)\ =\ \Big\{\,x\in\Lambda(n)\ :\ x\stackrel{\omega}{\longleftrightarrow}\partial\Lambda(n)\,\Big\}\,.$$
The set of the open clusters in the configuration~$\omega$ will be denoted
$$\mathcal{C}_n(\omega)\ =\ \Big\{\,C(x)\ :\ x\in\Lambda(n)\,\Big\}\,,$$
while~$\Clusdec(\omega)$ \correct{and~$\Clusdec(k,\,\omega)$ indicate, respectively,} the set of the open clusters which do not touch the boundary \correct{and the subset of these clusters which contain exactly~$k$ vertices, namely}
\begin{align*}
\Clusdec(\omega)
\ &=\ \Big\{\,C(x)\ :\ x\in\Lambda(n)\text{ such that }x\centernot\longleftrightarrow\partial\Lambda(n)\,\Big\}\,,\\
\Clusdec(k)\ &=\ \Clusdec(k,\,\omega)\ =\ \Big\{\,C\in\Clusdec(\omega)\ :\  \abs{C}=k\,\Big\}\,.
\end{align*}
Finally, let us note~$k^0(\omega)=\abs{\mathcal{C}_n(\omega)}$ and~$k^1(\omega)=\abs{\Clusdec(\omega)}+1$.

\subsection{The Ising model}

\label{section_def_Ising}

Fix~$n\geqslant 1$ and~$T>0$. The Ising model in the box~$\Lambda(n)$ at temperature~$T$ and with boundary condition~$+$ is defined as the probability measure
$$\mu_{n,\,T}^+\ :\ \sigma\in\acc{-,+}^{\Lambda(n)}\ \longmapsto\ \frac{1}{Z_{n,\,T}^+}\exp\left(-\frac{\mathcal{H}_n^+(\sigma)}{T}\right)$$
where~$Z_{n,\,T}^+$ is the normalization constant, and the Hamiltonian is given by
\begin{equation}
\label{def_Hamiltonien_Ising}
\mathcal{H}_n^+(\sigma)\ =\ -\sum_{\acc{x,y}\in\En}{\sigma(x)\sigma(y)}\ =\ -\demi\sum_{\substack{x,y\in\Lambda(n)\\ x\sim y}}{\sigma(x)\sigma(y)}
\end{equation}
if~$\sigma(x)=+$ for all~$x\in\partial\Lambda(n)$, and~$\mathcal{H}_n^+(\sigma)=+\infty$ otherwise.
We extend the above definition to the case of zero-temperature by setting
$$\mu_{n,\,0}^+\ :\ \sigma\in\acc{-,+}^{\Lambda(n)}\ \longmapsto\ \left\{\begin{aligned}
&1\quad\text{if }\sigma(x)=+\text{ for all }x\in\Lambda(n)\,,\\
&0\quad\text{otherwise.}
\end{aligned}\right.$$
The magnetization of a configuration~$\sigma:\Lambda(n)\rightarrow\acc{-,+}$ is defined by
$$m(\sigma)\ =\ \sum_{x\in\Lambda(n)}{\sigma(x)}\,.$$
We write~$T_c=2/\ln\left(\sqrt{2}+1\right)$ for the critical temperature of the Ising model.

\subsection{The random-cluster model}

The random-cluster model on~$\Lambda(n)$ with parameters~$p\in[0,1]$ and~$q>0$ and with boundary conditions~$\xi\in\acc{0,1}$ is defined as:
$$\phi_{n,\,p,\,q}^\xi(\omega)\ :\ \omega\in\acc{0,1}^{\En}\ \longmapsto\ \frac{1}{Z_{n,\,p,\,q}^\xi}\,q^{k^\xi(\omega)}\prod_{e\in\En}{p^{\omega(e)}(1-p)^{1-\omega(e)}}$$
where ~$Z_{n,\,p,\,q}^\xi$ is the appropriate normalization constant.
If~$p\in[0,1]$,~$q\geqslant 1$ and~$\xi\in\acc{0,1}$, then when~$n\correct{\to\infty}$, the measure~\smash{$\phi_{n,\,p,\,q}^\xi$} weakly converges to a probability distribution~$\phi_{p,\,q}^\xi$ on~\smash{$\acc{0,1}^{\Ed}$} (see theorem~4.19 in~\cite{GrimmettFK}). We write, for~$p\in[0,1]$,~$q\geqslant 1$ and~$\xi\in\acc{0,1}$,
$$\theta^\xi(p,\,q)\ =\ \phi_{p,\,q}^\xi\Big(\,\abs{C(0)}=\infty\,\Big)\,,$$
\correct{and the critical point of the model is defined as}
$$p_c(q)\ =\ \inf\Big\{\,p\in[0,1]\ :\ \theta^1(p,\,q)>0\,\Big\}\,.$$
We will sometimes omit the parameter~$q$ in the notation, which will mean that~$q=2$. In this particular case of~$q=2$, the measures~$\phi_{p,\,2}^0$ and~$\phi_{p,\,2}^1$ turn out to be equal (by corollary~3 in~\cite{Raoufi2017translationinvariant}), thus we will just write~$\theta(p)=\theta^1(p,\,2)$.

\subsection{Edwards-Sokal coupling}

\label{thm_Edwards_Sokal}
The random-cluster model with~$q=2$ \correct{can be coupled with} the Ising model \correct{(see section~1.4 of~\cite{GrimmettFK}), with the following consequence :

\begin{proposition}
\label{lemme_Edwards_Sokal}
Let~$T\geqslant 0$ and~$p=1-e^{-2/T}$ (with the convention~$e^{-2/0}=0$). If~$\omega\sim\Probap$, and if we let~$\sigma(x)=+$ for all~$x\in\Mn(\omega)$ and we assign constant spins on the other clusters in~$\omega$, independent between different clusters, each spin being equally distributed on~$\acc{-,+}$, then we have~$\sigma\sim\mu_{n,\,T}^+$.
\end{proposition}}

\subsection{Duality}

\label{section_dualite}

The planar random-cluster model enjoys a useful duality property that we \correct{recall} here. \correct{For}~$n\geqslant 2$, we define the dual of the box~$\Lambda(n)$ to be
$$\Lambda^\star(n)\ =\ \Lambda(n-1)+(-1)^{n+1}\left(\demi,\,\demi\right)\,,$$
in such a way that all the vertices of~$\Lambda^\star(n)$ lie in the middle of the faces of the graph~$(\Lambda(n),\,\En)$.
Let us also define the dual edges to be
$$\En^\star
\ =\ \mathbb{E}_{n-1}+(-1)^{n+1}\left(\demi,\,\demi\right)
\ =\ \Big\{\,\acc{x,\,y}\subset\Lambda^\star(n)\ :\ \norme{x-y}_1=1\,\Big\}\,.$$
The interior edges of the box~$\Lambda(n)$ are
\begin{equation}
\label{definition_aretes_interieures}
\En^{int}\ =\ \Big\{\,\acc{x,\,y}\in\En\ :\ \acc{x,\,y}\not\subset\partial\Lambda(n)\,\Big\}\,.
\end{equation}
For any interior edge~$e\in\En^{int}$, we write~$e^\star$ for the edge of~$\En^\star$ which intersects~$e$ perpendicularly in its middle\correct{, and if~$F\subset\En^{int}$ is a set of interior edges, we define~$F^\star=\acc{e^\star\,:\,e\in F}$.}
We then have~\smash{$\En^\star=\big(\En^{int}\big)^\star$}.
To any configuration~\smash{$\omega\in\acc{0,1}^{\En}$} \correct{is associated} a dual configuration~\smash{$\omega^\star\in\acc{0,1}^{\En^\star}$} given by~$\omega^\star(e^\star)=1-\omega(e)$.
If the graph~$(\Lambda^\star(n),\,\En^\star)$ is identified with the box~$(\Lambda(n-1),\,\mathbb{E}_{n-1})$, we obtain a configuration~\smash{$\omega^\star\in\acc{0,1}^{\mathbb{E}_{n-1}}$}. Following equation~6.12 in~\cite{GrimmettFK}, we have:

\begin{proposition}
\label{thm_dualite}
Take~$p\in[0,1]$,~$q\geqslant 1$ and~$n\geqslant 2$. Let~\smash{$\omega\in\acc{0,1}^{\En}$} be distributed according to~$\Probapq$. Then, the associated dual configuration~\smash{$\omega^\star$} is distributed according to~$\phi_{n-1,\,p^\star,\,q}^0$, where~$pp^\star=q(1-p)(1-p^\star)$.
\end{proposition}

\section{Exponential decay for supercritical temperatures}
\label{bigsec_surcritique}

The aim of this section is to prove the following result:

\begin{lemma}
\label{Ising_majo_surcritique}
For all~$a>3/2$, we have
$$\forall\,T_0>T_c\quad
\forall A>0\qquad
\limsupn\,\frac{1}{n}\,\sup_{T\geqslant T_0}\ln\mu_{n,\,T}^+\Big(\,\abs{m}\geqslant An^a\,\Big)\ <\ 0\,.$$
\end{lemma}

We will use estimates for the subcritical random-cluster model \correct{about} the number of vertices connected to the boundary and the size of the clusters. We will then show in section~\ref{section_Ising_majo_surcritique} how to deduce from these estimates a control of the magnetization in the Ising model with supercritical temperatures.

\subsection{Exponential decay of cluster sizes}

The \correct{next result follows from theorem~1.2 of~\cite{SharpPhaseTransition} and theorem~5.86 in~\cite{GrimmettFK}.}

\begin{lemma}
\label{ExpDecayClusterSize}
For any~$q\geqslant 1$ and~$p<p_c(q)$, there exists~$\psi(p,\,q)>0$ such that
$$\forall n,\,k\geqslant 1\quad
\forall v\in\Lambda(n)\qquad
\phi_{n,\,p,\,q}^0\Big(\,\abs{C(v)}\geqslant k\,\Big)\ \leqslant\ e^{-\psi k}\,.$$
\end{lemma}

\subsection{Control of the number of vertices connected to the boundary}

\begin{lemma}
\label{Mn_FK_majo_souscritique}
For all~$a>1$, we have the following upper bound:
$$\forall q\geqslant 1\quad
\forall p<p_c(q)\quad
\forall A>0\qquad
\limsupn\,\frac{1}{n^a}\ln\Probapq\Big(\,\abs{\Mn}\geqslant An^a\,\Big)\ <\ 0\,.$$
\end{lemma}
\begin{proof}
Let~$a>1$,~$q\geqslant 1$,~$p<p_c(q)$,~$A>0$ and~$n\geqslant 1$. By the nesting property of random-cluster measures (lemma~4.13 in~\cite{GrimmettFK}), we have
\begin{gather*}
\Probapq\Big(\,\abs{\Mn}\geqslant An^a\,\Big)
\ =\ \phi_{n+2,\,p,\,q}^0\Big(\,\abs{\Mn}\geqslant An^a\ \Big|\ \mathcal{E}_n\,\Big)\,,\numberthis\label{decay_Mn_cond}\\
\text{where}\qquad\mathcal{E}_n\ =\ \Big\{\,\forall e\in\mathbb{E}_{n+2}\backslash\En\quad \omega(e)=1\,\Big\}\,.
\end{gather*}
Yet, all the vertices of the boundary~$\partial\Lambda(n)$ are connected by paths using only edges of~$\mathbb{E}_{n+2}\backslash\En$. Therefore, fixing an arbitrary vertex~$v_0\in\partial\Lambda(n)$, we have
$$\phi_{n+2,\,p,\,q}^0\Big(\,\big\{\abs{\Mn}\geqslant An^a\big\}\cap\mathcal{E}_n\,\Big)
\ \leqslant\ \phi_{n+2,\,p,\,q}^0\Big(\,\abs{C(v_0)}\geqslant An^a\,\Big)
\ \leqslant\ e^{-\psi An^a}\,,$$
where~$\psi=\psi(p,\,q)$ is the constant given by lemma~\ref{ExpDecayClusterSize}.
\correct{Combining this with~(\ref{decay_Mn_cond}) and using the finite-energy property (see theorem~3.1 in~\cite{GrimmettFK}), we get}
$$\correct{\Probapq\Big(\,\abs{\Mn}\geqslant An^a\,\Big)
\ \leqslant\ \frac{e^{-\psi An^a}}{\phi_{n+2,\,p,\,q}^0\big(\,\mathcal{E}_n\,\big)}}
\ \leqslant\ \left(\frac{p+q(1-p)}{p}\right)^{8n+4}e^{-\psi An^a}\,.$$
Given that~$a>1$, this \correct{implies the proclaimed result}.
\end{proof}

\subsection{The cost of a change of boundary conditions}

\begin{lemma}
\label{lemme_change_BC}
\correct{Let~$q\geqslant 1$,~$p\in[0,1]$,~$\xi\in\acc{0,1}$,~$n\geqslant 1$ and}~\smash{$\mathcal{A}\subset\acc{0,1}^{\En}$}. Then
$$\phi_{n,\,p,\,q}^\xi\big(\mathcal{A}\big)\ \leqslant\ q^{4n}\phi_{n,\,p,\,q}^{1-\xi}\big(\mathcal{A}\big)\,.$$
\end{lemma}

\begin{proof}
\correct{This inequality is a consequence of the fact that the number of clusters touching the boundary is bounded by~$\abs{\partial\Lambda(n)}\leqslant 4n$, which entails that}
$$\forall\omega\in\acc{0,1}^{\En}\qquad
k^1(\omega)\ \leqslant\ k^0(\omega)\ \leqslant\ k^1(\omega)+4n\,.$$
\correct{implying the result.}
\end{proof}

\subsection{Negative correlation between cluster sizes}

\label{section_ansatz_BK}

We prove here a correlation inequality between the cardinalities of pairwise disjoint clusters, which is in a way a surrogate of the BK inequality, which is missing in the random-cluster model. \correct{Following an approach similar to that of~\cite{BC96} or~\cite{BG13}, the idea is that, once an open cluster is explored, the rest of the box behaves like a FK model with free boundary condition on the closed border of the explored cluster. This allows us to derive a correlation inequality for increasing events occurring in disjoint clusters.}
\correct{Let~$n\geqslant 1$ and~$N\in\acc{1,\,\ldots,\,n^2}$. For any~$k_1,\,\ldots,\,k_N\in\N$ and any pairwise disjoint~$v_1,\,\ldots,\,v_N\in\Lambda(n)$, let}
$$\mathcal{Q}_N\big(v_1,\,\ldots,\,v_N,\,k_1,\,\ldots,\,k_N\big)
\ =\ \acc{\,\begin{array}{c}
\forall i\leqslant N\quad \abs{C(v_i)}\geqslant k_i\,,\\
\forall i\neq j\quad v_i\centernot\longleftrightarrow v_j
\end{array}\,}\,.$$

\begin{lemma}
\label{Mn_FK_lemme_ansatz_BK}
Let~$q\geqslant 1$,~$p\in[0,1]$,~$\xi\in\acc{0,1}$,~$n\geqslant 1$ and~$N\in\acc{1,\,\ldots,\,n^2}$. For any pairwise distinct~$v_1,\,\ldots,\,v_N\in\Lambda(n)$ and for any~$k_1,\,\ldots,\,k_N\in\N$, we have
$$\Probapqxi\Big(\,\mathcal{Q}_N\big(v_1,\,\ldots,\,v_N,\,k_1,\,\ldots,\,k_N\big)\,\Big)
\ \leqslant\ \prod_{i=1}^N{\,\Probapqxi\Big(\,\abs{C\left(v_i\right)}\geqslant k_i\,\Big)}\,.$$
\end{lemma}

\begin{proof}
Having fixed~$q\geqslant 1$,~$p\in[0,1]$,~$\xi\in\acc{0,1}$ and~$n\geqslant 1$, we proceed by induction on~$N$.
The result is straightforward for~$N=1$. Let~$N\in\acc{1,\,\ldots,\,n^2-1}$ be such that the inequality holds for~$N$. Let~$v_1,\,\ldots,\,v_{N+1}\in\Lambda(n)$ be pairwise disjoint vertices and let~$k_1,\,\ldots,\,k_{N+1}\in\N$. Let us consider pairwise disjoint connected subsets~$C_1,\,\ldots,\,C_{N+1}\subset\Lambda(n)$ such that~$v_i\in C_i$ for every~$i\in\acc{1,\,\ldots,\,N+1}$. We consider the set
$$F\ =\ \partial^e\Bigg(\,\bigcup_{i=1}^N{C_i}\,\Bigg)\ =\ \acc{\,\acc{x,\,y}\in\En\ :\ x\in\bigcup_{i=1}^N{C_i}\ \text{ and }\ y\in\Lambda(n)\backslash\bigcup_{i=1}^N{C_i}\,}\,.$$
If~$C(v_i)=C_i$ for all~$i\leqslant n$, then all the edges of~$F$ must be closed, \correct{whence}
\begin{align*}
&\Probapqxi\Big(\,\forall i\leqslant N+1\quad C(v_i)=C_i\,\Big)
\ =\ \Probapqxi\Big(\,\forall i\leqslant N\quad C(v_i)=C_i\,\Big)\\
&\qquad\times\Probapqxi\Big(\,C(v_{N+1})=C_{N+1}\ \Big|\ \omega\restrict{F}=0\ \,\text{and}\ \,\forall i\leqslant N\quad C(v_i)=C_i\,\Big)\\
&\ =\ \Probapqxi\Big(\,\forall i\leqslant N\quad C(v_i)=C_i\,\Big)
\times\Probapqxi\Big(\,C(v_{N+1})=C_{N+1}\ \Big|\ \omega\restrict{F}=0\,\Big)\,,
\end{align*}
\correct{where we have used that, conditionally on the fact that~$\omega\equiv 0$ on~$F$, the events~$\big\{C(v_{N+1})=C_{N+1}\big\}$ and~$\big\{\forall i\leqslant N\quad C(v_i)=C_i\big\}$
are independent under~\smash{$\Probapqxi$} (in the vocabulary of~\cite{BC96}, the event~\smash{$\big\{\omega\restrict{F}=0\big\}$} is a decoupling event for these two events).}
Summing over all the connected sets~$C_{N+1}$ such that
$$v_{N+1}\ \in\ C_{N+1}\ \subset\ \Lambda(n)\,\backslash\,\bigcup_{i=1}^N{C_i}
\qquadet
\abs{C_{N+1}}\ \geqslant\ k_{N+1}\,,$$
\correct{and then using the FKG inequality, one gets}
\begin{align*}
&\Probapqxi\Big(\,\abs{C(v_{N+1})}\geqslant k_{N+1}\quadet \forall i\leqslant N\quad C(v_i)=C_i\,\Big)\\
&\quad \leqslant\ \Probapqxi\Big(\,\forall i\leqslant N\quad C(v_i)=C_i\,\Big)\,\Probapqxi\Big(\,\abs{C(v_{N+1})}\geqslant k_{N+1}\ \Big|\ \omega\restrict{F}=0\,\Big)\\
&\quad \leqslant\ \Probapqxi\Big(\,\forall i\leqslant N\quad C(v_i)=C_i\,\Big)\,\Probapqxi\Big(\,\abs{C(v_{N+1})}\geqslant k_{N+1}\,\Big)\,.
\end{align*}
\correct{Summing over all connected and pairwise disjoint~$C_1,\,\ldots,\,C_{N}\subset\Lambda(n)$ such that~$v_i\in C_i$ and~$\abs{C_i}\geqslant k_i$ for every~$i\leqslant N$ concludes the induction step.}
\end{proof}

\subsection{Control of the size of the clusters}

We obtain here an exponential inequality on the tail of the distribution of the cardinalities of the clusters, uniformly on any segment included in~$[0,\,p_c)$.

\begin{lemma}
\label{lemme_intermediaire_tailles_clusters}
For~$A,\,b,\,c\geqslant 0$ and~$n\geqslant 1$, we have the upper bound
\begin{gather*}
\forall q\geqslant 1\quad
\forall p_0<p_c(q)\quad
\forall A>0\quad
\forall b> 0\quad
\forall c\geqslant 0\\
\limsupn\,\frac{1}{n^{c}}\,\sup\limits_{p\leqslant p_0}\,\ln\Probapqf\Big(\,\mathcal{D}_n(A,\,b,\,c)\,\Big)\ <\ 0\,,\\
\text{where}\qquad
\mathcal{D}_n(A,\,b,\,c)
\ =\ \Bigg\{\,\sum_{C\in\mathcal{C}_n\,:\,\abs{C}\geqslant n^b}{\abs{C}}\geqslant A n^c\,\Bigg\}\,.
\end{gather*}
\end{lemma}

\begin{proof}
Let~$q\geqslant 1$,~$p\leqslant p_0<p_c(q)$,~$A>0$,~$b> 0$,~$c\geqslant 0$ and~$n\geqslant 1$. Defining~\smash{$M=\Ceil{An^{c-b}}$}, \correct{we have}
$$\mathcal{D}_n(A,\,b,\,c)\subset\bigcup_{N=1}^M{\bigcup_{\substack{v_1,\,\ldots,\,v_{N}\in\Lambda(n)\\ \text{pairwise distinct}}}{\bigcup_{\substack{n^b\leqslant k_1,\,\ldots,\,k_N\leqslant n^2\\ k_1+\cdots+k_N\geqslant A n^c}}{\mathcal{Q}_N\big(v_1,\,\ldots,\,v_{N},\,k_1,\,\ldots,\,k_N\big)}}}\,,$$
where~$\mathcal{Q}_N$ is the event defined in section~\ref{section_ansatz_BK}. Using lemma~\ref{Mn_FK_lemme_ansatz_BK}, it follows that
\begin{align*}
\Probapqf\Big(\,&\mathcal{D}_n(A,\,b,\,c)\,\Big)\\
\ &\leqslant\ \sum_{N=1}^M{\,\sum_{\substack{v_1,\,\ldots,\,v_{N}\in\Lambda(n)\\ \text{pairwise distinct}}}{\,\sum_{\substack{n^b\leqslant k_1,\,\ldots,\,k_N\leqslant n^2\\ k_1+\cdots+k_N\geqslant A n^c}}{\ \prod_{i=1}^{N}{\,\Probapqf\Big(\,\abs{C(v_i)}\geqslant k_i\,\Big)}}}}\\
\ &\leqslant\ M\left(n^2\right)^{2M}\exp\Big(\,-\psi A n^c\,\Big)\,,
\end{align*}
\correct{where~$\psi(p_0,\,q)$ comes from lemma~\ref{ExpDecayClusterSize}.}
This being true for all~$p\leqslant p_0$, we get
$$\frac{1}{n^c}\sup\limits_{p\leqslant p_0}
\ln\Probapqf\Big(\mathcal{D}_n(A,\,b,\,c)\Big)
\,\leqslant\,\frac{\ln M+4M\ln n}{n^c}-\psi A
\,=\,-\psi A+O\left(\frac{\ln n}{n^b}\right)\,.$$
\correct{Given that~$b>0$, this concludes our proof.}
\end{proof}

This inequality allows us to obtain a uniform control on the sum of the squares of the cardinalities of the clusters in the subcritical regime, when~$p$ does not get too close to~$p_c$. 
In practice, we could do without uniformity since this variable is an increasing variable (opening an edge connecting two clusters~$C_1$ and~$C_2$ increases this sum by~\smash{$(\abs{C_1}+\abs{C_2})^2-\abs{C_1}^2-\abs{C_2}^2\geqslant 0$}), but the uniformity in lemma~\ref{lemme_intermediaire_tailles_clusters} will be needed for the regime~$T<T_c$.

\begin{lemma}
\label{Mn_FK_controle_carres}
For~$a>3/2$, we have the following control in the subcritical regime:
$$\forall q\geqslant 1\ 
\forall p_0<p_c(q)\quad
\limn\frac{1}{n}\sup\limits_{p\leqslant p_0}\ln\Probapq\Bigg(\sum_{C\in\mathcal{C}_n}{\abs{C}^2}>2n^{a+1/2}\Bigg)\,=\,-\infty\,.$$
\end{lemma}

\begin{proof}
Let~$a>3/2$,~$q\geqslant 1$,~$p\leqslant p_0<p_c(q)$ and~$n\geqslant 1$. \correct{Using the notation of lemma~\ref{lemme_intermediaire_tailles_clusters}, if the event~$\mathcal{D}_n(1,\,a-3/2,\,a/2+1/4)$ does not occur, then we have}
\begin{align*}
\sum_{C\in\mathcal{C}_n}{\abs{C}^2}
\ &=\ \sum_{C\in\mathcal{C}_n\,:\,\abs{C}<n^{a-3/2}}{\abs{C}^2}\,+\,\sum_{C\in\mathcal{C}_n\,:\,\abs{C}\geqslant n^{a-3/2}}{\abs{C}^2}\\
\ &\leqslant\ n^{a-3/2}\times n^2\,+\,\left(n^{a/2+1/4}\right)^2\phantom{\demi}
\ =\ 2n^{a+1/2}\,.
\end{align*}
We can deduce that
\begin{align*}
\Probapqf\Bigg(\,\mathcal{D}_n\left(1,\ a-\frac{3}{2},\ \frac{a}{2}+\frac{1}{4}\right)\,\Bigg)
\ &\geqslant\ \Probapqf\Bigg(\,\sum_{C\in\mathcal{C}_n}{\abs{C}^2}\,>\,2n^{a+1/2}\,\Bigg)\\
\ &\geqslant\ \frac{1}{q^{4n}}\Probapq\Bigg(\,\sum_{C\in\mathcal{C}_n}{\abs{C}^2}\,>\,2n^{a+1/2}\,\Bigg)\,,
\end{align*}
\correct{where we have used lemma~\ref{lemme_change_BC} to change the boundary conditions.
Besides, given that~$a/2+1/4>1$, lemma~\ref{lemme_intermediaire_tailles_clusters} ensures that}
$$\limn\,\frac{1}{n}\,\sup\limits_{p\leqslant p_0}\,\ln\Probapqf\Bigg(\,\mathcal{D}_n\left(1,\,a-\frac{3}{2},\,\frac{a}{2}+\frac{1}{4}\right)\,\Bigg)\ =\ -\infty\,.$$
\correct{yielding the desired result.}
\end{proof}

\subsection{Moving to the Ising model}

\label{section_Ising_majo_surcritique}

We are now in a position to prove the exponential decay result above~$\correct{T_c}$

\begin{proof}[Proof of lemma~\ref{Ising_majo_surcritique}]
Let~$a>3/2$,~$A>0$ and~$T\geqslant T_0>T_c$. To handle the magnetization in the Ising model at temperature~$T$, we use the Edwards-Sokal coupling (see \correct{proposition~\ref{lemme_Edwards_Sokal}}). Thus, we set~\smash{$p=1-e^{-2/T}$} and~\smash{$p_0=1-e^{-2/T_0}$}, so that~$p\leqslant p_0<p_c(2)$, \correct{and we take~$\omega\sim\Probap$}. For each~$C\subset\Lambda(n)$, we draw~$\varepsilon_C$ equally distributed on~$\acc{-,+}$, the variables~$(\varepsilon_C)_{C\subset\Lambda(n)}$ being mutually independent and independent of~$\omega$. This represents many more variables than necessary, but it makes notations more concise. We write~$\Proba$ for the joint law of~$\omega$ and~$(\varepsilon_C)_{C\subset\Lambda(n)}$, \correct{and we define the spin configuration}
$$\sigma\ :\ x\in\Lambda(n)\ \longmapsto\ \left\{\begin{aligned}
& +\text{ if }x\in\Mn(\omega)\,,\\
&\ \varepsilon_C\text{ if }x\in C\in\Clusdec(\omega)\,.
\end{aligned}\right.$$
The configuration~$\sigma$ is then distributed according to~$\mu_{n,\,T}^+$, and \correct{we have}
$$m(\sigma)\ =\ \abs{\Mn(\omega)}+\sum_{C\in\Clusdec(\omega)}{\abs{C}\varepsilon_C}\,.$$
Therefore, \correct{by conditioning on~$\sum_{C\in\Clusdec}{\abs{C}^2}$, we can write}
\begin{align*}
\mu_{n,\,T}^+&\Big(\,\abs{m}\geqslant An^a\,\Big)
\ \leqslant\ \Probap\left(\,\abs{\Mn}\,\geqslant\,\frac{A}{2}n^a\,\right)\\
&+\Probap\Bigg(\,\sum_{C\in\Clusdec}{\abs{C}^2}\,>\,2n^{a+1/2}\,\Bigg)\\
&+\Proba\left(\,\Bigg|\sum_{C\in\Clusdec(\omega)}{\abs{C}\varepsilon_C}\Bigg|\,\geqslant\,\frac{A}{2}n^a\ \biggg|\ \sum_{C\in\Clusdec}{\abs{C}^2}\,\leqslant\,2n^{a+1/2}\,\right)\,.\numberthis\label{eq8448}
\end{align*}
It follows from Hoeffding's inequality (see~\cite{Hoeffding}) that
$$\Proba\left(\,\Bigg|\sum_{C\in\Clusdec(\omega)}{\abs{C}\varepsilon_C}\Bigg|\,\geqslant\,\frac{A}{2}n^a\ \biggg|\ \sum_{C\in\Clusdec}{\abs{C}^2}\,\leqslant\,2n^{a+1/2}\,\right)
\ \leqslant\ 2\exp\left(-\frac{A^2}{16}n\right)\,,$$
where we have used the fact that~$a>3/2$.
\correct{Plugging this in~(\ref{eq8448}) and taking the supremum over~$T\geqslant T_0$, we obtain}
\begin{multline*}
\sup_{T\geqslant T_0}\mu_{n,\,T}^+\Big(\,\abs{m}\geqslant An^a\,\Big)
\ \leqslant\ \phi_{n,\,p_0,\,2}^1\left(\,\abs{\Mn}\,\geqslant\,\frac{A}{2}n^a\,\right)\\
+\sup\limits_{p\leqslant p_0}\Probap\Bigg(\,\sum_{C\in\Clusdec}{\abs{C}^2}\,>\,2n^{a+1/2}\,\Bigg)
+2\exp\left(-\frac{A^2}{16}n\right)\,.
\end{multline*}
Combining this with lemmas~\ref{Mn_FK_majo_souscritique} and~\ref{Mn_FK_controle_carres} yields the desired result.
\end{proof}

\section{Exponential decay for subcritical temperatures}
\label{bigsec_souscritique}

The goal of this section is to prove the following estimate:

\begin{lemma}
\label{Ising_majo_souscritique}
For all~$a<2$, we have
$$\forall\,T_0<T_c\quad
\forall A>0\qquad
\limsupn\,\frac{1}{n}\,\sup_{T\leqslant T_0}\,\ln\mu_{n,\,T}^+\Big(\,m\leqslant An^a\,\Big)\ <\ 0\,.$$
\end{lemma}

For a fixed temperature~$T<T_c$, this follows from theorem~5.2 in~\cite{Wulff}, but we need a control which is uniform on any segment included in~$[0,\,T_c)$.

\subsection{Control of the number of vertices connected to the boundary}

\correct{The following} counterpart of lemma~\ref{Mn_FK_majo_souscritique} in the supercritical regime \correct{is a direct consequence of theorem~5.5 of~\cite{Wulff} \correct{and} lemma~\ref{ExpDecayClusterSize}.}
\begin{lemma}
\label{Mn_FK_majo_surcritique}
We have the following upper bound:
$$\forall q\geqslant 1\quad
\forall p>p_c(q)\qquad
\limsupn\,\frac{1}{n}\,\ln\Probapq\left(\,\abs{\Mn}\leqslant\frac{\theta(p)n^2}{2}\,\right)\ <\ 0\,.$$
\end{lemma}

\subsection{Control of the size of the clusters which do not touch the boundary}

\correct{We want a uniform control on any segment included in~$[0,\,T_c)$ of the variable}
$$\sum_{C\in\Clusdec\,:\,\abs{C}\geqslant n}{\abs{C}}
\ =\ \abs{\Big\{\,x\in\Lambda(n)\,\backslash\,\Mn\ :\ \abs{C(x)}\geqslant n\,\Big\}}\,.$$
To this end, we use duality (see paragraph~\ref{section_dualite}) to convert large clusters which do not touch the boundary into large connected contours in the dual configuration. This will allow us to use the estimate given by lemma~\ref{lemme_intermediaire_tailles_clusters} on cluster sizes in the regime~$p<p_c(q)$. 
This is why lemma~\ref{lemme_intermediaire_tailles_clusters} was stated with~$\mathcal{C}_n$ rather than~$\Clusdec$, even though a control of~\correct{$\Clusdec$} would have been enough to obtain lemma~\ref{Ising_majo_surcritique}.

\begin{lemma}
\label{Mn_FK_controle_carres_surcritique}
We have the following estimate:
\begin{gather*}
\forall q\geqslant 1\quad
\forall p_0>p_c(q)\quad
\forall A>0\\
\limsupn\,\frac{1}{n}\,\sup\limits_{p\geqslant p_0}\,\ln\Probapq\biggg(\,\sum_{C\in\mathcal{C}_n^-\,:\,\abs{C}\geqslant n}{\abs{C}}\geqslant An^2\,\biggg)\ <\ 0\,.
\end{gather*}
\end{lemma}

\begin{proof}
Let~$q\geqslant 1$,~$p\geqslant p_0>p_c(q)$,~$A>0$,~$n\geqslant 2$. 
\correct{and~\smash{$B=\big\{C\in\Clusdec:\abs{C}\geqslant n\big\}$}.
For any cluster~$C\in B$, we define its \guillemets{external boundary}, denoted by~$\partial^{ext}C$, to be the set of the edges of~$\partial^e C$ which connect a vertex of~$C$ to a vertex in the infinite connected component of~$\Z^2\backslash C$.}
Because the clusters~$C\in B$ do not touch the boundary of the box, we have~$\partial^{ext} C\subset\En^{int}$, where~$\En^{int}$ was defined by~(\ref{definition_aretes_interieures}). \correct{The} dual~\smash{$\big(\partial^{ext}C\big)^\star$} of this external boundary \correct{then} forms a connected closed contour surrounding~$C$, \correct{implying} that
\begin{equation}
\label{dextCstar}
\abs{\big(\partial^{ext}C\big)^\star}
\ =\ \abs{\partial^{ext}C}
\ \geqslant\ 1+\diam\, C
\ \geqslant\ \sqrt{\abs{C}}
\ \geqslant\ \sqrt{n}\,.
\end{equation}
We now define the set of edges
$$F\ =\ \bigcup_{C\in B}{\partial^{ext} C}\,.$$
\correct{Recalling that one edge can connect at most two different clusters,~(\ref{dextCstar}) implies}
\begin{equation}
\label{cardF}
\abs{F}
\ \geqslant\ \demi\sum_{C\in B}{\abs{\partial^{ext} C}}
\ \geqslant\ \demi\sum_{C\in B}{\sqrt{\abs{C}}}
\ \geqslant\ \demi\sqrt{\sum_{C\in B}{\abs{C}}}\,.
\end{equation}
The connected components of~$F^\star$ are unions of~$(\partial^{ext} C)^\star$ for a certain number of clusters~$C\in B$, which implies, given~(\ref{dextCstar}), that they all contain at least~$\sqrt{n}$ vertices.
\correct{What's more, we have~$\omega=0$ on~$F$, implying that~$\omega^\star=0$ on~$F^\star$.
Given~(\ref{cardF}), }all this leads to
$$\sum_{C\in B}{\abs{C}}\ \geqslant\ An^2
\quadimplique
\abs{F^\star}
\ \geqslant\ \frac{\sqrt{A}}{2}n
\quadimplique
\omega^\star\in\mathcal{D}_{n-1}\left(\frac{\sqrt{A}}{2},\ \demi,\ 1\right)\,.$$
\correct{Therefore, the result follows from lemma~\ref{lemme_intermediaire_tailles_clusters}, using duality (proposition~\ref{thm_dualite}).} 
\end{proof}

\subsection{Moving to the Ising model}

We prove here the uniform exponential decay for the Ising model.

\begin{proof}[Proof of lemma~\ref{Ising_majo_souscritique}]
Let~$a<2$,~$T\leqslant T_0<T_c$ and~$A>0$. Let~$p=1-e^{-2/T}$ and~$p_0=1-e^{-2/T_0}$, so that~$p\geqslant p_0>p_c(2)$.
\correct{With~$\Proba$ as in section~\ref{section_Ising_majo_surcritique},} we have
\begin{align*}
\mu_{n,\,T}^+\Big(\,m\leqslant An^a\,\Big)
\ &\leqslant\ \Probap\left(\,\abs{\Mn}\,\leqslant\,\frac{\theta(p_0)n^2}{2}\,\right)\\
&\qquad\quad+\Proba\left(\,\sum_{C\in\Clusdec(\omega)}{\abs{C}\varepsilon_C}\,\leqslant\,-\frac{\theta(p_0)n^2}{2}+An^a\,\right)\\
\ &\leqslant\ \phi_{n,\,p,\,2}^1\left(\,\abs{\Mn}\,\leqslant\,\frac{\theta(p_0)n^2}{2}\,\right)\\
+\Proba\Biggg(\,\sum_{\substack{C\in\Clusdec(\omega)\\ \abs{C}< n}}{\abs{C}\varepsilon_C}&\,\leqslant\,-\frac{\theta(p_0)n^2}{8}\,\Biggg)
+\Proba\Biggg(\,\sum_{\substack{C\in\Clusdec(\omega)\\ \abs{C}\geqslant n}}{\abs{C}\varepsilon_C}\,\leqslant\,-\frac{\theta(p_0)n^2}{8}\,\Biggg)\,,
\end{align*}
\correct{for~$n$ large enough. Now note that}
$$\sum_{\substack{C\in\Clusdec(\omega)\\ \abs{C}< n}}{\abs{C}^2}
\ \leqslant\ n\times\sum_{\substack{C\in\Clusdec(\omega)\\ \abs{C}< n}}{\abs{C}}
\ \leqslant\ n\abs{\Lambda(n)}
\ =\ n^3\,,$$
which implies by Hoeffding's inequality that
$$\Proba\biggg(\,\sum_{C\in\Clusdec(\omega)\,:\,\abs{C}< n}{\abs{C}\varepsilon_C}\,\leqslant\,-\frac{\theta(p_0)n^2}{8}\,\biggg)
\ \leqslant\ \exp\left[\,-\frac{2}{4n^3}\left(\frac{\theta(p_0)n^2}{8}\right)^2\,\right]\,.$$
\correct{The result then follows by} taking the supremum over~$T\leqslant T_0$ \correct{and using} the exponential estimates of lemmas~\ref{Mn_FK_majo_surcritique} and~\ref{Mn_FK_controle_carres_surcritique}.
\end{proof}

\section{Finite-size scaling results for near-critical FK-percolation}

\label{bigsec_FSS}

This section is devoted to the proof of proposition~\ref{prop_messikh}, \correct{which states that, for every~$s\in(0,\,8/41)$ and}~$K,\,\delta>0$ and for any sequence~$p_n\in[0,1]$ \correct{such that~\smash{$p_n-p_c(2)\eqninfty K n^{-s}$},} we have
\begin{subequations}
\begin{eqnarray}
&\limn\Probapn\Big(\,\abs{\Mn}\leqslant(1+\delta)\theta(p_n)\abs{\Lambda(n)}\,\Big)\ =\ 1\,,\label{condition1}\\
&\limn\Probapn\left(\,\max\limits_{C\in\mathcal{C}_n^-}\abs{C}\leqslant n^{s+1/2}\,\right)\ =\ 1\,,\label{condition2}\\
&\limn\Probapn\Big(\,\abs{\Mn\cap\Lambda(n_1)}\geqslant(1-\delta)\theta(p_n)\abs{\Lambda(n_1)}\,\Big)\ =\ 1\,,\label{condition3}
\end{eqnarray}
\end{subequations}
where~$n_1=\Ent{5n/6}$.
From Onsager~\cite{Onsager} and Yang~\cite{Yang}, we know that
\begin{equation}
\label{equivalent_theta}
\theta(p)\ \stackrel{p\downarrow p_c(2)}{\sim}\ \left[\,8\left(\frac{p}{p_c(2)}-1\right)\,\right]^{1/8}\,.
\end{equation}

\begin{proof}[Proof of proposition~\ref{prop_messikh}]
Let~$s$ be such that~$0<s<8/41$, let~$K,\,\delta>0$ and~$p_n\in[0,1]$ such that~\smash{$p_n-p_c(2)\sim K n^{-s}$}. \correct{Let~$a'$ be} such that
\begin{equation}
\label{encadrement_a_prime2}
5\ <\ a'\ <\ \frac{1}{s}-\frac{1}{8}\,.
\end{equation}
\souspreuve{Proof of~(\ref{condition1}):}
\correct{It follows from} theorem~2 of~\cite{CerfMessikh2011} with~$a'$,~$p=p_n$ and~$\delta$ that
$$\limsupn\,\frac{1}{(p_n-p_c(2))^{2a'+1/4}n^2}\,\ln\Probapn\Big(\,\abs{\Mn}>(1+\delta)\theta(p_n)\abs{\Lambda(n)}\,\Big)\ <\ 0\,.$$
Yet, we know by~(\ref{encadrement_a_prime2}) that~$2a'+1/4<2/s$
which implies that
$$\big(p_n-p_c(2)\big)^{2a'+1/4}n^2
\ \eqninfty\ K^{2a'+1/4}\times n^{2-(2a'+1/4)s}
\ \cvninfty\ +\infty\,.$$
Therefore, we have \correct{proved}~(\ref{condition1}).

\suitepreuve{Proof of~(\ref{condition2}):}
A cluster~$C\subset\Lambda(n)$ is said to \guillemets{cross} a sub-box~$B\subset\Lambda(n)$ if there exists in~$C\cap B$ an open path from the bottom side to the top side of~$B$ and an open path from the left side to the right side of~$B$ (this condition is in fact stronger than the one of~\cite{CerfMessikh2011}, which only requests~$C$ to intersect all the faces of~$B$, but the same proof works with our definition).
\correct{Using theorem~1 of~\cite{CerfMessikh2011} with~$a'$ and~\smash{$M=\Ent{n^{s/2+1/4}}$}, we obtain}
\begin{equation}
\label{label4912}
\limn\phi_{m_n,\,p_n,\,2}^1\big(\mathcal{E}_n\big)
\ =\ 1\,,
\end{equation}
where~$m_n=\Ent{6n/5}$, and where~$\mathcal{E}_n$ is the event \guillemets{in the box~$\Lambda(n)$, there exists a cluster~$C_0$ crossing every sub-box of~$\Lambda(n)$ with diameter~$M$}.
\correct{By} the nesting property of random-cluster measures \correct{and the FKG inequality}, we have
\begin{equation}
\label{eq24796}
\Probapn\big(\mathcal{E}_n\big)
\ =\ \phi_{m_n,\,p_n,\,2}^1\left(\,\mathcal{E}_n\ \Big|\ \omega\restrict{\mathbb{E}_{m_n}\!\backslash\En}=1\,\right)
\ \geqslant\ \phi_{m_n,\,p_n,\,2}^1\big(\mathcal{E}_n\big)\,.
\end{equation}
If this event~$\mathcal{E}_n$ occurs, then every open path of diameter at least~$M$ must be included into~$C_0$. Yet the cluster~$C_0$ touches the boundary~$\partial\Lambda(n)$, so all open paths of diameter larger than~$M$ are connected to~\correct{$\partial\Lambda(n)$}, \correct{whence}
$$\mathcal{E}_n
\quad\Rightarrow\quad
\max\limits_{C\in\Clusdec}\abs{C}
\correct{\ \leqslant\ \max\limits_{C\in\Clusdec}\big(1+\diam\,C\big)^2}
\ \leqslant\ M^2\ \leqslant\ n^{s+1/2}\,.$$
\correct{Combining this with~(\ref{label4912}) and~(\ref{eq24796}), we obtain}~(\ref{condition2}).

\suitepreuve{Proof of~(\ref{condition3}):}
\correct{We now} choose~$\alpha$ such that
$$a's\ <\ \alpha\ <\ \frac{8a'}{8a'+1}\,.$$
By theorem~3 of~\cite{CerfMessikh2011} applied with~$a'$,~$\alpha$~and~$\delta$, we have
$$\Probapn\left(\,\max_{C\in\mathcal{C}_{n_1}}\abs{C}\geqslant(1-\delta)\theta(p_n)\abs{\Lambda(n_1)}\,\right)\ \cvninfty\ 1\,.$$
Combining this with~(\ref{condition2}), we get
\begin{equation}
\label{Mn_FK_eq_Cmax2}
\Probapn\left(\,\max_{C\in\mathcal{C}_{n_1}}\abs{C}\geqslant(1-\delta)\theta(p_n)\abs{\Lambda(n_1)}
\ \text{and}\ \max\limits_{C\in\Clusdec}\abs{C}\leqslant n^{s+1/2}\,\right)\,\cvninfty\,1\,.
\end{equation}
Yet, according to the asymptotics for~$\theta(p)$ given by~(\ref{equivalent_theta}), we know that
$$(1-\delta)\theta(p_n)\abs{\Lambda(n_1)}
\ \eqninfty\ (1-\delta)\left(\frac{5}{6}\right)^2
\left(\frac{8}{p_c(2)}\right)^{1/8}n^{2-s/8}\,.$$
As~$s<4/3$, we have~$2-s/8>s+1/2$ and thus, for~$n$ large enough, we have
~$(1-\delta)\theta(p_n)\abs{\Lambda(n_1)}
>n^{s+1/2}$.
Hence, if the event in~(\ref{Mn_FK_eq_Cmax2}) occurs, then there is in~$\Lambda(n_1)$ a cluster containing strictly more than~$n^{s+1/2}$ vertices, which must therefore be connected to~$\partial\Lambda(n)$. Thus, the result follows from~(\ref{Mn_FK_eq_Cmax2}).
\end{proof}

\section{Lower bound on the partition function}

\label{bigsec_minoZn}

The goal of this section is to show the following lower bound on the partition function of our model, \correct{following the strategy presented in paragraph~\ref{section_heuristique_Zn}.}

\begin{lemma}
\label{Ising_lemme_minoration_Z_n}
For every~$a\in(31/16,\,2)$ such that~$\mathcal{FSS}(16-8a)$ holds, we have
$$\liminfn\,\frac{\ln Z_n}{(\ln n)n^{\rho}}\ >\ -\infty
\qquadou
\rho\ =\ \max\left(\frac{a}{2},\ \frac{33}{2}-8a\right)\,.$$
\end{lemma}

\subsection{The price for closing edges}

We start by stating a useful lemma to estimate the probability for closing a random set of edges, \correct{which is a direct consequence of lemma~6.3 in~\cite{CerfPisztora2000}.}

\begin{lemma}
\label{lemme_fermeture_H}
Let~$n\geqslant 1$,~$p\in]0,1[$,~$q\geqslant 1$ and~\smash{$\mathcal{A}\subset\acc{0,1}^{\En}$}. Let~$H$ be an arbitrary application which associates to any configuration~$\omega\in\mathcal{A}$ a certain set of edges~$H(\omega)\subset\En$. We consider the application~\correct{$\psi:\omega\in\mathcal{A}\mapsto \omegaH$. Then
$$\Probapq\big(\psi(\mathcal{A})\big)\ \geqslant\ \Probapq\big(\mathcal{A}\big)\left[\frac{p\wedge(1-p)}{3p\abs{\Lambda(n)}}\right]^N
\quadou
N\ =\ \max_{\omega\in\mathcal{A}}\abs{H(\omega)}\,.
$$}
\end{lemma}

\subsection{Construction of the fixed point and preliminary estimates}

\label{section_quasi_surcritique}

\correct{Following the strategy presented in paragraph~\ref{section_heuristique_Zn}, we define}
$$b'_n\,=\,n^a\sqrt{2}\Bigg[-\ln\left(1-p_c(2)-\frac{3^8 p_c(2)}{8n^{16-8a}}\right)\Bigg]^{-1/2}
\ \text{and}\ 
b_n\,=\,\Ent{b'_n}-\mathbb{1}_{\Ent{b'_n}\centernot{\equiv} n^2\,[2]}\,,$$
so that we always have
\begin{equation}
\label{bn_mod2}
b_n\ \equiv\ n^2\ \modulo{2}\,.
\end{equation}
Our final aim is to get a lower bound on the probability that, at temperature~$T=b_n^2/n^{2a}$, the magnetization is exactly equal to~$b_n$. For this, we need~$b_n$ to have the same parity as~$\abs{\Lambda(n)}=n^2$, hence the above definition.
We also set~$p_n=\varphi_n(b_n)$, where the function~$\varphi_n$ is \correct{given} by~(\ref{definition_varphi}).
Recall that, if~$\omega$ is a percolation configuration in the box~$\Lambda(n)$, then~$\Clusdec(\omega)$ denotes the set of the open clusters in~$\omega$ which do not touch the boundary~$\partial\Lambda(n)$, and~$\Clusdec(1,\,\omega)$ denotes the subset of these clusters which contain only one vertex. The present section is devoted to the proof of the following statement, which is a direct outcome of the assumptions~$\mathcal{FSS}(16-8a)$.

\begin{lemma}
\label{Mn_FK_proba_Fn}
Let~$a\in(31/16,\,2)$. Setting~$n_1=\Ent{5n/6}$, we define the event
\begin{equation}
\label{Mn_FK_definition_Gn}
\mathcal{G}_n\ =\ \acc{\,\begin{array}{c}
\abs{\Mn}\leqslant 4n^a\,,\quad
\abs{\Mn\cap\Lambda\left(n_1\right)}\geqslant 2n^a\,,\\
\max\limits_{C\in\Clusdec}\abs{C}\leqslant n^{33/2-8a}\leqslant \abs{\Clusdec(1)}-1
\end{array}\,}\,.
\end{equation}
If the hypotheses~$\mathcal{FSS}(16-8a)$ hold, then we have~\smash{$\limn\Probapn\big(\mathcal{G}_n\big)=1$}.
\end{lemma}

Before proving this lemma, we check that our definition of~$b_n$ leads to the right convergence speed towards the critical point.

\begin{lemma}
For any~$a>16/9$, we have the following estimates:
\begin{align}
b_n\ &\eqninfty\ n^a\sqrt{T_c}\,,\label{Mn_FK_eq_bn}\\
p_n-p_c(2)\ &\eqninfty\ \frac{3^8 p_c(2)}{8n^{16-8a}}\,,\label{Mn_FK_equivalent_pns}\\
\phantom{\Big]}\theta(p_n)\abs{\Lambda(n)}\ &\eqninfty\ 3n^a\,.\label{eq_theta_pns}
\end{align}
\end{lemma}

\begin{proof}
Equation~(\ref{Mn_FK_eq_bn}) is a consequence of the definition of~$b_n$.
To show~(\ref{Mn_FK_equivalent_pns}), note on the one hand that,~$\varphi_n$ being decreasing on~\smash{$\acc{0,\,\ldots,\,n^2}$}, we have
$$p_n\ =\ \varphi_n(b_n)\ \geqslant\ \varphi_n\left(b'_n\right)
\ =\ p_c(2)+\frac{3^8 p_c(2)}{8n^{16-8a}}\,.$$
On the other hand, we can write
\begin{multline*}
p_n-\left(p_c(2)+\frac{3^8 p_c(2)}{8n^{16-8a}}\right)
\ \leqslant\ \varphi_n\left(b'_n-2\right)-\varphi_n\left(b'_n\right)\\
=\ e^{-2n^{2a}/{b'_n}^2}-e^{-2n^{2a}/(b'_n-2)^2}
\ \leqslant\ \frac{2n^{2a}}{\left(b'_n-2\right)^2}-\frac{2n^{2a}}{{b'_n}^2}
\ =\ \frac{2n^{2a}\left(4b'_n-4\right)}{\left(b'_n-2\right)^2{b'_n}^2}\,.
\end{multline*}
Combining this with the asymptotics for~$b_n$ given by equation~(\ref{Mn_FK_eq_bn}), we get
$$p_n-p_c(2)-\frac{3^8 p_c(2)}{8n^{16-8a}}
\ =\ O\left(\frac{1}{n^a}\right)
\ =\ o\left(\frac{1}{n^{16-8a}}\right)\,,$$
\correct{where we have used the fact that~$a>16/9$.}
Therefore, equation~(\ref{Mn_FK_equivalent_pns}) is satisfied, and~(\ref{eq_theta_pns}) can be deduced from it, using the expansion of~$\theta(p)$ given by~(\ref{equivalent_theta}).
\end{proof}

We now prove that~\correct{$\mathcal{FSS}(16-8a)$ implies} the result about the event~$\mathcal{G}_n$.

\begin{proof}[Proof of lemma~\ref{Mn_FK_proba_Fn}]
Let~$a\in(31/16,\,2)$ such that~$\mathcal{FSS}(16-8a)$ holds.
\correct{Using the expansion given by~(\ref{eq_theta_pns}) and applying~(\ref{condition1}) with~$\delta=1/6$ leads to}
\begin{equation}
\label{Mn_FK_dev_sup_step14}
\Probapn\Big(\,\abs{\Mn}\leqslant 4n^a\,\Big)
\ \geqslant\ \Probapn\left(\,\abs{\Mn}\leqslant\frac{7}{6}\theta(p_n)\abs{\Lambda(n)}\,\right)\ \cvninfty\ 1\,.
\end{equation}
\correct{Using now~(\ref{condition3}) with~$\delta=1/50$ and again the expansion~(\ref{eq_theta_pns}), we get}
\begin{equation}
\label{Mn_FK_eq_Cmax_2}
\limn\Probapn\Big(\,\abs{\Mn\cap\Lambda(n_1)}\geqslant 2n^a\,\Big)
\ =\ 1\,.
\end{equation}
Eventually, we show a lower bound on the number of clusters of size~$1$ which do not touch the boundary of the box. These clusters will be useful when passing to the Ising model (in section~\ref{section_minoration_Zn}), in order to tune precisely the value of the magnetization by choosing the spins of these unit clusters. Following theorem~3.21 in~\cite{GrimmettFK}, the random-cluster measure is stochastically dominated by the corresponding Bernoulli percolation measure, that is to say~\smash{$\Probapn\preceq\phi_{n,\,p_n,\,1}^1$}. \correct{The variable~$\abs{\Clusdec(1)}$ being decreasing, this entails that}
\begin{equation}
\label{comp_stoch}
\Probapn\Big(\abs{\Clusdec(1)}\geqslant 1+n^{33/2-8a}\Big)
\,\geqslant\,\phi_{n,\,p_n,\,1}^1\Big(\abs{\Clusdec(1)}\geqslant 1+n^{33/2-8a}\Big)\,.
\end{equation}
Considering one half of the vertices inside the box, we define
\begin{equation}
\label{somme_clusters_1}
U_n\ =\ U_n(\omega)\ =\ \sum_{\substack{x\in\Lambda(n)\backslash\partial\Lambda(n)\\ \norme{x}_1\,\equiv\, 0\ \modulo{2}}}{\mathbb{1}_{C(x)=\acc{x}}}\,,
\end{equation}
which is such that~$\abs{\Clusdec(1)}\geqslant U_n$. Taking the expectation, we can deduce that
$$\phi_{n,\,p_n,\,1}^1\big(U_n\big)\ \geqslant\ \Ent{\frac{(n-2)^2}{2}}(1-p)^4
\ \eqninfty\ \frac{n^2(1-p)^4}{2}\,.$$
Therefore, using the fact that~$a>31/16>29/16$, we have
$$1+n^{33/2-8a}\ =\ o\left(n^2\right)
\ =\ o\Big(\phi_{n,\,p_n,\,1}^1\big(U_n\big)\Big)\,.$$
Yet, under the law~$\phi_{n,\,p_n,\,1}^1$, the variables appearing in the sum~(\ref{somme_clusters_1}) are mutually independent, whence by Hoeffding's inequality,
$$\limn\,\phi_{n,\,p_n,\,1}^1\Big(\,U_n< 1+n^{33/2-8a}\,\Big)
\ =\ 0\,.$$
\correct{Using that~$\abs{\Clusdec(1)}\geqslant U_n$ and recalling~(\ref{comp_stoch}), this leads to}
$$\limn\,\Probapn\Big(\,\abs{\Clusdec(1)}\geqslant 1+n^{33/2-8a}\,\Big)\ =\ 1\,.$$
\correct{Combining this with~(\ref{Mn_FK_dev_sup_step14}),~(\ref{Mn_FK_eq_Cmax_2}) and~(\ref{condition2})}, we obtain lemma~\ref{Mn_FK_proba_Fn}.
\end{proof}

\subsection{Surgery on the set of vertices connected to the boundary}

\label{section_chirurgie}

We explain here the surgery step on~$\Mn$, starting from~$\omega\in\mathcal{G}_n$.

\begin{lemma}
\label{lemme_proba_Rn}
Let~$K>0$, and let us consider the event
\begin{equation}
\label{definition_Rn}
\mathcal{R}_n=\acc{\,\begin{aligned}
&\exists\,H\subset\En,\ \ \abs{H}\leqslant Kn^{a/2},\ \ \abs{\Mn(\omega_H)}=\Ceil{\frac{\abs{\Mn(\omega)}+b_n}{2}}\\
&\quad\qquadet \max\limits_{C\in\Clusdec}\abs{C}\leqslant n^{33/2-8a}\leqslant \abs{\Clusdec(1)}-1
\end{aligned}\,}\,.
\end{equation}
For all~$a\in(31/16,\,2)$ such that the postulates~$\mathcal{FSS}(16-8a)$ hold, with~$p_n$ defined as in section~\ref{section_quasi_surcritique}, there exists~$K>0$ such that~\smash{$\limn\Probapn\big(\mathcal{R}_n\big)=1$}.
\end{lemma}

\begin{proof}
Let~$a\in(31/16,\,2)$ be such that~$\mathcal{FSS}(16-8a)$ holds. Let~$n\geqslant 12$ and~$n_1=\Ent{5n/6}$, let~$p_n$ be defined as in section~\ref{section_quasi_surcritique} and~$\omega\in\mathcal{G}_n$.
We are going to construct a set of edges~$H=H(\omega)\subset\En$ such that
$$\abs{\Mn(\omega_{H})}\ =\ \Ceil{\frac{\abs{\Mn(\omega)}+b_n}{2}}\,.$$
For every~$j\geqslant 1$, we consider the edge set~$E_j=\partial^e\Lambda(j)$. If~$j,\,k\geqslant 1$ are such that~$\abs{j-k}\geqslant 2$, then~$E_j\cap E_k=\emptyset$. From this we deduce by the pigeonhole principle that there exists an integer~$j(\omega)$ satisfying
$$n_1\ \leqslant\ 2j(\omega)\ \leqslant\ n-2
\qquadet
\Big|\,E_{2j(\omega)}\cap\Edge{\Mn(\omega)}\Big|\ \leqslant\ \frac{\big|\,\Edge{\Mn(\omega)}\big|}{L_n}\,,$$
where
$$L_n\ =\ \Ent{\frac{n-2}{2}}-\Ceil{\frac{n_1}{2}}+1\ \eqninfty\ \frac{n}{12}\,.$$
We choose such an integer~$j(\omega)$ and we let~$H_0(\omega)=E_{2j(\omega)}\cap\Edge{\Mn(\omega)}$. Recall that, by the definition~(\ref{Mn_FK_definition_Gn}) of~$\mathcal{G}_n$, we have~$\abs{\Mn(\omega)}\leqslant 4n^a$, which entails that
$$\abs{H_0(\omega)}
\ \leqslant\ \frac{2\abs{\Mn(\omega)}}{L_n}
\ \leqslant\ \frac{8n^a}{L_n}\,.$$
Again by the definition of the event~$\mathcal{G}_n$, we know that
$$\abs{\Mn(\omega)}
\ \geqslant\ \abs{\Mn(\omega)\cap\Lambda(n_1)}
\ \geqslant\ 2n^a
\ >\ b_n+1\,,$$
for~$n$ large enough (because~$b_n+1\sim n^a\sqrt{T_c}$ and~$\sqrt{T_c}<2$), \correct{implying that}
$$\abs{\Mn(\omega)}\ >\ \Ceil{\frac{\abs{\Mn(\omega)}+b_n}{2}}\,.$$
Thus, we may consider a subset~$H_1(\omega)\subset H_0(\omega)$, maximal in the sense of inclusion among the subsets satisfying
\begin{equation}
\label{Mn_FK_wH1}
\abs{\Mn\big(\omega_{H_1(\omega)}\big)}\ \geqslant\ \Ceil{\frac{\abs{\Mn(\omega)}+b_n}{2}}\,.
\end{equation}
Due to the fact that the set~$H_0(\omega)$ separates~$\Lambda(n_1)$ from~$\partial\Lambda(n)$, we have
$$\abs{\Mn\big(\omega_{H_0(\omega)}\big)}\ \leqslant\ \abs{\Mn(\omega)}-\abs{\Mn(\omega)\cap\Lambda\left(n_1\right)}\ \leqslant\ \abs{\Mn(\omega)}-2n^a\,.$$
Yet, because~$\omega$ realizes the event~$\mathcal{G}_n$, we have~$\abs{\Mn(\omega)}\leqslant4n^a$, whence
$$\abs{\Mn\big(\omega_{H_0(\omega)}\big)}\ \leqslant\ \frac{\abs{\Mn(\omega)}}{2}+\frac{4n^a}{2}-2n^a\ =\ \frac{\abs{\Mn(\omega)}}{2}
\ <\ \Ceil{\frac{\abs{\Mn(\omega)}+b_n}{2}}\,.$$
It follows that the inclusion~$H_1(\omega)\subset H_0(\omega)$ is strict, and hence,~$H_1(\omega)$ being maximal, there exists an edge~$e\in H_0(\omega)\backslash H_1(\omega)$ such that
\begin{equation}
\label{Mn_FK_cond_e}
\abs{\Mn\big(\omega_{H_1(\omega)\cup\acc{e}}\big)}\ <\ \Ceil{\frac{\abs{\Mn(\omega)}+b_n}{2}}\,.
\end{equation}
Thus, closing the edge~$e$ in~$\omega_{H_1(\omega)}$ strictly decreases~$\abs{\Mn}$, hence one of the endpoints of~$e$, \correct{say}~$v$, must end up disconnected from~$\partial\Lambda(n)$ when closing~$e$ in~$\omega_{H_1(\omega)}$. Let us write~$C_v$ for the cluster of~$v$ in~$\omega_{H_1(\omega)\cup\acc{e}}$, and let~$E_v$ be the set of the edges of~$\Edge{C_v}$ which are open in~$\omega_{H_1(\omega)\cup\acc{e}}$. Then, because~$C_v$ is the piece which is disconnected when closing the edge~$e$ in~$\omega_{H_1(\omega)}$, we have
\begin{equation}
\label{Mn_FK_Cv}
\abs{C_v}\ =\ \abs{\Mn\big(\omega_{H_1(\omega)}\big)}-\abs{\Mn\big(\omega_{H_1(\omega)
\cup\acc{e}}\big)}\,.
\end{equation}
\begin{figure}[ht]
\begin{center}
\begin{tikzpicture}[scale=0.85]
\filldraw[gray!50, draw=black] (0,0) --++ (3,0) --++ (0,3) --++ (-3,0) -- cycle;
\filldraw[fill=white] (2,0.4) to[out=40,in=220] (2.3,1.4) to[out=40,in=240] (2.6,2.1) to[out=60,in=-45] (2.8,2.85) to[out=135,in=30] (2.4,2.6) to[out=210,in=-45] (1.4,2.6) to[out=135,in=70] (0.15,2.7) to[out=250,in=105] (0.4,1.9) to[out=285,in=45] (1.3,1.2) to[out=225,in=130] (1,0.4) to[out=310,in=180] (1.5,0.15) to[out=0,in=220] (2,0.4);
\draw (1.5,-0.6) node{$\Mn(\omega)$};
\draw[magenta,dashed] (0.6,0.6) to (2.4,0.6) to (2.4,2.4) to (0.6,2.4) to (0.6,0.6);
\draw[magenta] (1.2,2.1) node{$\Lambda(n_1)$};
\draw[->,thick] (3.5,1.5) --++ (1,0);
\draw (4,1) node{close~$H_0$};

\begin{scope}[xshift=5cm]
\filldraw[gray!50, draw=black] (0,0) --++ (3,0) --++ (0,3) --++ (-3,0) -- cycle;
\filldraw[fill=white] (2,0.4) to[out=0,in=180] (2.6,0.4) to[out=90,in=-90] (2.6,2.1) to[out=60,in=-45] (2.8,2.85) to[out=135,in=30] (2.4,2.6) to[out=180,in=0] (1.4,2.6) to[out=135,in=70] (0.15,2.7) to[out=250,in=105] (0.4,1.9) to[out=-90,in=90] (0.4,0.4) to[out=0,in=180] (1,0.4) to[out=310,in=180] (1.5,0.15) to[out=0,in=220] (2,0.4);
\draw[red, ultra thick] (2,0.4) to (2.6,0.4) to (2.6,2.1) (2.4,2.6) to (1.4,2.6) (0.4,1.9) to (0.4,0.4) to (1,0.4);
\draw (1.5,-0.6) node{$\Mn\big(\omega_{H_0(\omega)}\big)$};
\draw[red] (2.25,1.1) node{$H_0$};
\end{scope}

\draw[->,thick] (8.5,1.5) --++ (1,0);
\draw (9,1) node{reopen};
\draw (9,0.5) node{$H_0\backslash H_1$};

\begin{scope}[xshift=10cm]
\filldraw[gray!50, draw=black] (0,0) --++ (3,0) --++ (0,3) --++ (-3,0) -- cycle;
\filldraw[fill=white] (2,0.4) to[out=0,in=180] (2.6,0.4) to[out=90,in=-90] (2.6,0.8) to[out=210,in=220] (2.3,1.4) to[out=40,in=145] (2.6,1.4) to[out=90,in=-90] (2.6,2.1) to[out=60,in=-45] (2.8,2.85) to[out=135,in=30] (2.4,2.6) to[out=180,in=0] (1.4,2.6) to[out=135,in=70] (0.15,2.7) to[out=250,in=105] (0.4,1.9) to[out=-90,in=90] (0.4,1.7) to[out=-20,in=155] (0.6,1.6) to[out=-25,in=45] (1.2,1.2) to[out=225,in=20] (1,1.1) to[out=200,in=60] (0.4,0.8) to[out=-90,in=90] (0.4,0.4) to[out=0,in=180] (1,0.4) to[out=310,in=180] (1.5,0.15) to[out=0,in=220] (2,0.4);
\draw[dotted,ultra thick,blue] (2.6,0.8) to (2.6,1.4) (0.4,1.7) to (0.4,0.8);
\draw (1.5,-0.6) node{$\Mn\big(\omega_{H_1(\omega)}\big)$};
\draw[blue] (1.5,2.1) node{$H_0\backslash H_1$};
\draw [blue,decoration={markings,mark=at position 1 with
    {\arrow[scale=2,>=stealth]{>}}},postaction={decorate}] (1.6,1.8) to (2.5,1.1);
\draw [blue,decoration={markings,mark=at position 1 with
    {\arrow[scale=2,>=stealth]{>}}},postaction={decorate}] (1.4,1.8) to (0.5,1.3);
\end{scope}

\draw[->,thick] (13.5,1.5) --++ (1,0);
\draw (14,1) node{close~$H_2$};

\begin{scope}[xshift=15cm]
\filldraw[gray!50, draw=black] (0,0) --++ (3,0) --++ (0,3) --++ (-3,0) -- cycle;
\filldraw[fill=white] (2,0.4) to[out=0,in=180] (2.6,0.4) to[out=90,in=-90] (2.6,0.8) to[out=210,in=220] (2.3,1.4) to[out=40,in=145] (2.6,1.4) to[out=90,in=-90] (2.6,2.1) to[out=60,in=-45] (2.8,2.85) to[out=135,in=30] (2.4,2.6) to[out=180,in=0] (1.4,2.6) to[out=135,in=70] (0.15,2.7) to[out=250,in=105] (0.4,1.9) to[out=-90,in=90] (0.4,1.7) to[out=-20,in=155] (0.6,1.6) to[out=-120,in=170] (1,1.1) to[out=200,in=60] (0.4,0.8) to[out=-90,in=90] (0.4,0.4) to[out=0,in=180] (1,0.4) to[out=310,in=180] (1.5,0.15) to[out=0,in=220] (2,0.4);
\draw[red,ultra thick] (0.6,1.6) to[out=-120,in=170] (1,1.1);
\draw (1.5,-0.6) node{$\Mn\left(\omega_{H(\omega)}\right)$};
\draw[red] (1,1.5) node{$H_2$};
\end{scope}
\end{tikzpicture}
\end{center}
\caption{Steps of the surgical procedure to go from~$\abs{\Mn}=b_n+x$ to~\smash{$\abs{\Mn}=b_n+x/2$}.}
\end{figure}

\noindent Defining now
$$m\ =\ \Ceil{\frac{\abs{\Mn(\omega)}+b_n}{2}}-\abs{\Mn\big(\omega_{H_1(\omega)\cup\acc{e}}\big)}\,,$$
it follows from equations~(\ref{Mn_FK_wH1}),~(\ref{Mn_FK_cond_e}) and~(\ref{Mn_FK_Cv}) that~$1\leqslant m\leqslant\abs{C_v}$.
Therefore, according to lemma~1 of~\cite{ModelsofSOCinPerco}, there exists a subset~$H_2(\omega)\subset E_v$ with
$$\abs{H_2(\omega)}\ \leqslant\ K_0\sqrt{\abs{C_v}}\,,$$
where~$K_0>0$ is a fixed constant, such that the connected component of~$v$ in the graph~$\big(C_v,\,E_v\backslash H_2(\omega)\big)$ contains exactly~$m$ vertices. We then consider the set of edges~$H(\omega)=H_1(\omega)\cup H_2(\omega)$, whose cardinality is
$$\abs{H(\omega)}
\ \leqslant\ \abs{H_0(\omega)}+\abs{H_2(\omega)}
\ \leqslant\ \frac{8n^a}{L_n}+K_0\sqrt{4 n^a}\\
\ =\ O\big(n^{a/2}\big)\,.$$
Now take~$K>0$ (independent of~$n$ and~$\omega$) such that, for every~$n\geqslant 1$,
$$\frac{8n^a}{L_n}+K_0\sqrt{4 n^a}
\ \leqslant\ Kn^{a/2}\,.$$
Then, we have~$\abs{H(\omega)}\leqslant K n^{a/2}$.
Besides, by construction of~$H$, we have
$$\abs{\Mn\big(\omegaH\big)}\ =\ \abs{\Mn\big(\omega_{H_1(\omega)\cup\acc{e}}\big)}+m\ =\ \Ceil{\frac{\abs{\Mn(\omega)}+b_n}{2}}\,.$$
At the end of the day, we have proved the inclusion~$\mathcal{G}_n\subset\mathcal{R}_n$ and, consequently, the result follows from lemma~\ref{Mn_FK_proba_Fn}.
\end{proof}

What do we get by closing the edges of~$H(\omega)$ for~$\omega\in\mathcal{R}_n$? We will show that the resulting configurations~$\omegaH$ fall into the event
\begin{equation}
\label{Mn_FK_definition_Sn}
\mathcal{S}_n\ =\ \acc{\,\begin{aligned}
&\exists\,\Czero\subset\Clusdec\ :\ \abs{\Czero}\leqslant 2Kn^{a/2},\quad \abs{\Mn}-\sum_{C\in\Czero}{\abs{C}}=b_n\\
&\quad \quadet\max\limits_{C\in\Clusdec\backslash\Czero}\abs{C}\,\leqslant\, n^{33/2-8a}\,\leqslant\,\abs{\Clusdec(1)\backslash\Czero}
\end{aligned}\,}\,.
\end{equation}
The set~$\Czero$ corresponds to the piece that was disconnected from the boundary to \correct{go} half the way from~$\abs{\Mn}$ to~$b_n$, as explained in section~\ref{section_heuristique_Zn}. This event will allow us in section~\ref{section_minoration_Zn} to build an Ising configuration with~$m(\sigma)=b_n$, through the Edwards-Sokal coupling. To this end, we will force the clusters~$C\in\Czero$ to be assigned a negative spin and we will hope for the contribution of the other clusters to cancel out. \correct{The two last conditions appearing in~$\mathcal{S}_n$} will enable us to control this contribution. But first of all, we show the following estimate:

\begin{lemma}
\label{Mn_FK_lemme_minoration_Z_n}
For every~$a\in(31/16,2)$ satisfying~$\mathcal{FSS}(16-8a)$, with~$p_n$ defined as in section~\ref{section_quasi_surcritique}, there exists a constant~$K>0$ such that
$$\liminfn\,\frac{1}{(\ln n)n^{a/2}}\,\ln\Probapn\big(\mathcal{S}_n\big)\ >\ -\infty\,.$$
\end{lemma}

To construct a configuration realizing this event~$\mathcal{S}_n$, we start with~$\correct{\omega\in\mathcal{R}_n}$ and we close the set of edges~$H(\omega)$ given by the definition~(\ref{definition_Rn}) of~$\mathcal{R}_n$. This allows us to derive a lower bound on the probability of the event~$\mathcal{S}_n$, thanks to lemma~\ref{lemme_fermeture_H} which controls the price for closing edges.

\begin{proof}[Proof of lemma~\ref{Mn_FK_lemme_minoration_Z_n}]
Let~$a\in(31/16,2)$ such that~$\mathcal{FSS}(16-8a)$ holds, let~$p_n$ defined as in section~\ref{section_quasi_surcritique}, and take~$K>0$ given by lemma~\ref{lemme_proba_Rn}, such that
\begin{equation}
\label{limRn}
\limn\,\Probapn\big(\mathcal{R}_n\big)\ =\ 1\,.
\end{equation}
Let~$\correct{\omega\in\mathcal{R}_n}$. By the definition~(\ref{definition_Rn}) of~$\mathcal{R}_n$, we can take~$H(\omega)\subset\En$ such that
$$\abs{H(\omega)}\ \leqslant\ Kn^{a/2}
\qquadet
\abs{\Mn\big(\omegaH\big)}\ =\ \Ceil{\frac{\abs{\Mn(\omega)}+b_n}{2}}\,.$$
By shrinking the set~$H(\omega)$ if necessary, we can assume that~$H(\omega)\subset\Edge{\Mn(\omega)}$, which ensures that the clusters which do not touch the boundary in~$\omega$ are left undamaged in~$\omegaH$, so that~\smash{$\Clusdec(\omega)\subset\Clusdec\big(\omegaH\big)$}.
We then want to show that~$\correct{\omegaH\in\mathcal{S}_n}$. To this end, we have to build a set~\smash{$\Czero\subset\Clusdec\big(\omegaH\big)$} which satisfies the conditions of the definition~(\ref{Mn_FK_definition_Sn}) of~$\mathcal{S}_n$. A natural candidate is the set of the clusters in~$\omegaH$ which were connected to the boundary before the closure of the edges of~$H(\omega)$, that is to say in the configuration~$\omega$, but are not anymore in the modified configuration~$\omegaH$. Thus, we define
$$\Czero\ =\ \Czero(\omega)\ =\ \Clusdec\big(\omegaH\big)\backslash\Clusdec(\omega)\,.$$
Because we have taken~$H(\omega)\subset\Edge{\Mn(\omega)}$, the clusters of~$\Czero$ must be included in~$\Mn$, whence
$$\Czero\ =\ \Big\{\,C\in\Clusdec\big(\omegaH\big)\ :\ C\subset\Mn(\omega)\,\Big\}\,.$$
Therefore, we have
$$\bigcup_{C\in\Czero}{C}\ =\ \Mn(\omega)\,\backslash\,\Mn\big(\omegaH\big)\,,$$
which implies that
\begin{align*}
\abs{\Mn\big(\omegaH\big)}-\sum_{C\in\Czero}{\abs{C}}
\ &=\ 2\abs{\Mn\big(\omegaH\big)}-\abs{\Mn(\omega)}\\
\ &=\ 2\Ceil{\frac{\abs{\Mn(\omega)}+b_n}{2}}-\abs{\Mn(\omega)}\,,
\end{align*}
\correct{where we have used the fact that~$\omega\in\mathcal{R}_n$.}
If~$\abs{\Mn(\omega)}+b_n$ is even, then we have
$$\abs{\Mn\big(\omegaH\big)}-\sum_{C\in\Czero}{\abs{C}}
\ =\ b_n\,,$$
and
$$\abs{\Clusdec\left(1,\,\omegaH\right)\!\backslash\Czero}
\ =\ \abs{\Clusdec(1,\,\omega)}
\ \geqslant\ n^{33/2-8a}+1\,.$$
Assume now that~$\abs{\Mn(\omega)}+b_n$ is odd. In this case, we have
$$\abs{\Mn\big(\omegaH\big)}-\sum_{C\in\Czero}{\abs{C}}
\ =\ b_n+1\,.$$
Therefore, we need to add a unit cluster in~$\Czero$. By the definition of~$\mathcal{R}_n$, we have~$\abs{\Clusdec(1,\,\omega)}\geqslant 1$,
which allows us to choose a cluster~$C_1\in\Clusdec(1,\,\omega)$. We then let~$\Czero'=\Czero\cup \acc{C_1}$, and we have
$$\abs{\Mn\big(\omegaH\big)}-\sum_{C\in\Czero'}{\abs{C}}\ =\ b_n\,.$$
Besides, we have 
$$\abs{\Clusdec\left(1,\,\omegaH\right)\!\backslash\Czero'}
\ =\ \abs{\Clusdec(1,\,\omega)\backslash\acc{C_1}}
\ =\ \abs{\Clusdec(1,\,\omega)}-1
\ \geqslant\ n^{33/2-8a}\,.$$
This is why we gave ourselves a margin of~$1$ between the condition on the unit clusters in~$\mathcal{R}_n$ and the one in~$\mathcal{S}_n$.
Thus, if we set~$\Czero'=\Czero$ in the case where~$\abs{\Mn(\omega)}+b_n$ even, then whatever the parity of~$\abs{\Mn(\omega)}+b_n$, we have
$$\abs{\Mn\big(\omegaH\big)}-\sum_{C\in\Czero'}{\abs{C}}\ =\ b_n
\qquadet
\abs{\Clusdec\left(1,\,\omegaH\right)\!\backslash\Czero'}\ \geqslant\ n^{33/2-8a}\,.$$
In addition to that, due to the fact that~\smash{$\Clusdec\big(\omegaH\big)\backslash\Czero'\subset\Clusdec(\omega)$}, we know that
$$\max\limits_{C\in\Clusdec(\omegaH)\backslash\Czero'}\abs{C}
\ \leqslant\ \max\limits_{C\in\Clusdec(\omega)}\abs{C}
\ \leqslant\ n^{33/2-8a}\,.$$
Eventually, note that the closure of one edge cannot increase the number of open clusters by more than~$1$, whence~$\abs{\Czero}\leqslant\abs{H(\omega)}$ and thus
$$\abs{\Czero'}
\ \leqslant\ \abs{\Czero}+1
\ \leqslant\ \abs{H(\omega)}+1
\ \leqslant\ Kn^{a/2}+1
\ \leqslant\ 2Kn^{a/2}\,,$$
upon increasing~$K$ if necessary.
Thus, we have proved that the application~$\psi:\omega\in\mathcal{R}_n\mapsto\omega_{H(\omega)}$
takes its values in in~$\mathcal{S}_n$. Using lemma~\ref{lemme_fermeture_H}, we deduce that
$$\Probapn\big(\mathcal{S}_n\big)
\ \geqslant\ \Probapn\Big(\psi\big(\mathcal{R}_n\big)\Big)\\
\ \geqslant\ \Probapn\big(\mathcal{R}_n\big)\left(\frac{1-p_n}{3n^2}\right)^{2Kn^{a/2}}\,.$$
\correct{Combining this with~(\ref{limRn}) and using that~$p_n\rightarrow p_c(2)\in(0,1)$, we obtain}
$$\liminfn\,\frac{\ln\Probapn\big(\mathcal{S}_n\big)}{(\ln n)n^{a/2}}
\ \geqslant\ -4K
\ >\ -\infty\,,$$
which is the desired result.
\end{proof}

\subsection{Moving to the Ising model}

\label{section_minoration_Zn}

Armed with the estimate given by lemma~\ref{Mn_FK_lemme_minoration_Z_n}, we are now in a position to prove our lower bound on the partition function.

\begin{proof}[Proof of lemma~\ref{Ising_lemme_minoration_Z_n}]
Let~$a\in(31/16,\,2)$ be such that the finite-size scaling postulates~$\mathcal{FSS}(16-8a)$ hold.
Let~$p_n$ and~$b_n$ be defined as in section~\ref{section_quasi_surcritique}. Let us consider
$$\Tns\ =\ \frac{2}{-\ln\left(1-p_n\right)}\ =\ \frac{b_n^2}{n^{2a}}\,.$$
Recall that, by rewriting~$Z_n$ in the form~(\ref{reecriture_Zn}), we have seen that it suffices to prove a lower bound on the probability that, under the law~\smash{$\Probapn$}, the magnetization is exactly~$b_n$.
As in the proof of lemmas~\ref{Ising_majo_surcritique} and~\ref{Ising_majo_souscritique}, we use Edwards-Sokal coupling to deduce the lower bound on~$Z_n$ from our result on the random-cluster model. Let~$\omega\correct{\sim}\Probapn$,~$(\varepsilon_C)_{C\subset\Lambda(n)}$ \correct{and~$\sigma$ defined as in section~(\ref{section_Ising_majo_surcritique}).}
For every configuration~$\correct{\omega\in\mathcal{S}_n}$, we choose a set~$\Czero(\omega)\subset\Clusdec(\omega)$ satisfying the properties in the definition~(\ref{Mn_FK_definition_Sn}) of~$\mathcal{S}_n$.
This enables us to define the event
$$\mathcal{T}_n\ =\ \mathcal{S}_n\cap\acc{\,\sum_{C\in\Clusdec\backslash\Czero}{\abs{C}\varepsilon_C}=0
\quadet
\forall\,C\in\Czero\quad \varepsilon_C=-\,}\,.$$
Note that, if~$\mathcal{T}_n$ occurs, it implies that
$$m(\sigma)\ =\ \left(\,\abs{\Mn(\omega)}-\sum_{C\in\Czero}{\abs{C}}\,\right)+\sum_{C\in\Clusdec\backslash\Czero}{\abs{C}\varepsilon_C}
\ =\ b_n+0\ =\ b_n\,,$$
so that our lower bound~(\ref{reecriture_Zn}) on the partition function becomes
\begin{equation}
\label{lien_Tn_Zn}
Z_n\ \geqslant\ \Proba\Big(\,m(\sigma)=b_n\,\Big)
\ \geqslant\ \Proba\big(\mathcal{T}_n\big)\,.
\end{equation}
We fix now a configuration~$\omega_0\in\mathcal{S}_n$, and we will reason conditionally on the event~$\acc{\omega=\omega_0}$. In this context, the variables~$\Clusdec=\Clusdec(\omega_0)$ and~$\Czero=\Czero(\omega_0)$ are henceforth fixed. We can write
\begin{equation}
\label{Ising_eq_Tn}
\Proba\Big(\,\mathcal{T}_n\ \Big|\ \omega=\omega_0\,\Big)\ =\ \frac{1}{2^{\abs{\Czero}}}\,\Proba\Bigg(\,\sum_{C\in\Clusdec\backslash\Czero}{\abs{C}\varepsilon_C}=0\,\Bigg)\,.
\end{equation}
Let~$N=\Ent{n^{33/2-8a}}$. Recall that, by the definition~(\ref{Mn_FK_definition_Sn}) of~$\mathcal{S}_n$, we have
\begin{equation}
\label{Cardmax_CnmCz}
\max\limits_{C\in\Clusdec\backslash\Czero}\abs{C}
\ \leqslant\ N
\qquadet
\abs{\Clusdec(1)\backslash\Czero}\ \geqslant\ N\,,
\end{equation}
which means that there are at least~$N$ unit clusters in~$\Clusdec\backslash\Czero$. 
The idea is to leave~$N$ of these clusters aside, and to control the magnetization of the other clusters of~$\Clusdec\backslash\Czero$ to show that it falls within the range~$[-N,\,N]$ with sufficient probability. We will then be able to use these~$N$ unit clusters kept aside to adjust the value of the magnetization and to force it to reach exactly~$0$. Thus, we consider a set~$\Cun\subset\Clusdec(1)\backslash\Czero$, with cardinality~$N$.
To ensure that the overall magnetization of the clusters of~$\Clusdec\backslash(\Czero\cup\Cun)$ falls within the range~$[-N,\,N]$, we sort these clusters depending on their cardinality. Then, for~$j\in\acc{1,\,\ldots,\,N}$, we control separately the contribution of the clusters of~$\Clusdec\backslash(\Czero\cup\Cun)$ which contain exactly~$j$ vertices. Therefore, for every~$j\in\acc{0,\,\ldots,\,N}$, we write
$$S_j\ =\ \sum_{\substack{C\in\Clusdec\backslash(\Czero\cup\Cun)\\ \abs{C}\leqslant j}}{\abs{C}\varepsilon_C}
\ =\ \sum_{i=1}^j{\,\sum_{C\in\Clusdec(i)\backslash(\Czero\cup\Cun)}{\abs{C}\varepsilon_C}}\,.$$
It follows from~(\ref{Cardmax_CnmCz}) that
$$S_N\ =\ \sum_{C\in\Clusdec\backslash(\Czero\cup\Cun)}{\abs{C}\varepsilon_C}\,.$$
Equation~(\ref{Ising_eq_Tn}) therefore leads to
\begin{equation}
\label{Ising_Tn_cond}
\Proba\Big(\,\mathcal{T}_n\ \Big|\ \omega=\omega_0\,\Big)\ =\ \frac{1}{2^{\abs{\Czero}}}\,\Proba\Bigg(\,S_N+\sum_{C\in\Cun}{\varepsilon_C}=0\,\Bigg)\,.
\end{equation}
By Stirling's formula, we can fix~$K_2\in(0,1)$ such that
$$\forall k\geqslant 1\qquad
\binom{2k}{k}\frac{1}{4^k}\ \geqslant\ \frac{K_2}{\sqrt{2k}}\,.$$
We then prove by induction on~$j$ that, for all~$j\in\acc{0,\,\ldots,\,N}$,
\begin{equation}
\label{rec_local_limite}
\Proba\Big(\,\abs{S_j} \leqslant N\,\Big)
\ \geqslant\ \left(\frac{K_2}{2n}\right)^{j}\,.
\end{equation}
The result is obvious for~$j=0$, since~$S_0=0$. Let~$j\in\acc{1,\,\ldots,\,N}$ be such that~(\ref{rec_local_limite}) holds for~$j-1$. We consider~$B_j=\Clusdec(j)\backslash(\Czero\cup\Cun)$,
which is such that
$$S_{j}\ =\ S_{j-1}+j\sum_{C\in B_j}{\varepsilon_C}\,.$$
If~$\correct{B_j=\emptyset}$, then the inequality~(\ref{rec_local_limite}) follows from the induction hypothesis. Now suppose that~$\correct{B_j\neq\emptyset}$.
If~$\abs{B_j}$ is even, then to obtain a null contribution from the clusters~$C\in B_j$ to the total magnetization, we need \correct{exactly} half of these clusters to choose a~$+$ spin. Thus, for~$\abs{B_j}$ even, we have
$$\Proba\Big(\abs{S_{j}}\leqslant N\ \Big|\ \abs{S_{j-1}}\leqslant N\Big)
\ \geqslant\ \Proba\Bigg(\,\sum_{C\in B_j}{\varepsilon_C}=0\,\Bigg)
\ =\ \binom{\abs{B_j}}{\abs{B_j}/2}\frac{1}{2^{\abs{B_j}}}
\ \geqslant\ \frac{K_2}{n}\,.$$
Assume now that~$\abs{B_j}$ is odd, and choose an arbitrary cluster~$C_0\in B_j$. To control~$S_j$, we request a null overall contribution from the clusters of~$B_j\backslash\acc{C_0}$, and we ask for the extra term coming from~$C_0$ to have a sign opposed to~$S_{j-1}$, which ensures that~$\abs{S_{j}}$ remains in the interval~$[-N,\,N]$. Let us consider
$$\eta\,:\,x\in\Z\ \longmapsto\ -\mathrm{sgn}(x)\ =\ \left\{\begin{aligned}
&+\text{ if }x\leqslant 0\,,\\
&-\text{ otherwise.}
\end{aligned}\right.$$
We can write, for~$\abs{B_j}$ odd,
\begin{align*}
\Proba\Big(\,\abs{S_{j}}\leqslant N\ \Big|\ \abs{S_{j-1}}\leqslant N\,\Big)
&\ \geqslant\ \Proba\Bigg(\,\varepsilon_{C_0}=\eta(S_{j-1})\ \text{ and }\ \sum_{C\in B_j\backslash\acc{C_0}}{\varepsilon_C}=0\,\Bigg)\\
&\ =\ \demi\binom{\abs{B_j}-1}{\big(\abs{B_j}-1\big)/2}\frac{1}{2^{\abs{B_j}-1}}
\ \geqslant\ \frac{K_2}{2n}\,.
\end{align*}
This proves the induction step, and thus the lower bound~(\ref{rec_local_limite}) holds for every~$j\in\acc{0,\,\ldots,\,N}$, which implies in particular that
\begin{equation}
\label{S0_lambda}
\Proba\Big(\,\abs{S_N} \leqslant N\,\Big)
\ \geqslant\ \left(\frac{K_2}{2n}\right)^{N}\,.
\end{equation}
We will now use the variables~$\varepsilon_C$ for~$C\in \Cun$, which we had shelved aside, to compensate~$S_N$ and thus attain a magnetization exactly equal to~$b_n$. This is only possible if~$S_N$ has the same parity as the number~$N$ of unit clusters in~$\Cun$. Let us check that this is indeed the case, by writing
\begin{multline*}
S_N\ =\ \sum_{C\in\Clusdec\backslash(\Czero\cup\Cun)}{\abs{C}\varepsilon_C}
\ \equiv\ \sum_{C\in\Clusdec}{\abs{C}}+\sum_{C\in \Czero}{\abs{C}}+\abs{\Cun}\quad \modulo{2}\\
\ =\ \abs{\Lambda(n)}-\abs{\Mn}+\sum_{C\in \Czero}{\abs{C}}+N
\ =\ n^2-b_n+N\ \equiv\ N\quad\modulo{2}\,,
\end{multline*}
where we have used the fact that~$\omega\in\mathcal{S}_n$ and equation~(\ref{bn_mod2}) which ensures that~$b_n\equiv n^2\ \modulo{2}$. Hence, thanks to the precaution we took when defining~$b_n$,~$N-S_N$ is always even.
Thus, for every~$s\in\acc{0,\,\ldots,\,N}$, we write
$$\Proba\Bigg(\,\sum_{C\in \Cun}{\varepsilon_C}=-S_N\ \Bigg|\ S_N=N-2s\,\Bigg)
\ =\ \binom{N}{s}\frac{1}{2^N}
\ \geqslant\ \frac{1}{2^N}\,.$$
This being true for every~$s\in\acc{0,\,\ldots,\,N}$, it follows that
$$\Proba\Bigg(\,\sum_{C\in \Cun}{\varepsilon_C}=-S_N\ \Bigg|\ \abs{S_N}\leqslant N\,\Bigg)
\ \geqslant\ \frac{1}{2^N}\,.$$
From this we deduce, using~(\ref{S0_lambda}), that
$$\Proba\Bigg(\,S_N+\sum_{C\in\Cun}{\varepsilon_C}=0\,\Bigg)
\ \geqslant\ \frac{1}{2^N}\times\left(\frac{K_2}{2n}\right)^{N}
\ =\ \left(\frac{K_2}{4n}\right)^N\,.$$
Combining this with equation~(\ref{Ising_Tn_cond}), we obtain that
$$\Proba\Big(\,\mathcal{T}_n\ \Big|\ \omega=\omega_0\,\Big)
\ \geqslant\ \frac{1}{2^{\abs{\Czero}}}\left(\frac{K_2}{4n}\right)^N
\ \geqslant\ \left(\demi\right)^{2Kn^{a/2}}\left(\frac{K_2}{4n}\right)^{n^{33/2-8a}}\,.$$
This being true for all the configurations~$\correct{\omega_0\in\mathcal{S}_n}$, it follows that
$$\Proba\big(\mathcal{T}_n\,\big|\,\mathcal{S}_n\big)
\ \geqslant\ \left(\frac{K_2}{8n}\right)^{2Kn^{\rho}}
\qquadou
\rho\ =\ \max\left(\frac{a}{2},\ \frac{33}{2}-8a\right)\,.$$
We deduce that
\begin{equation}
\label{PGD_cond}
\frac{\ln\Proba\big(\mathcal{T}_n\,\big|\,\mathcal{S}_n\big)}{(\ln n)n^{\rho}}\ \geqslant\ \frac{2K(\ln K_2-\ln 8)}{\ln n}-2K\ \cvninfty\ -2K\,.
\end{equation}
To conclude the proof, it remains only to condition on~$\mathcal{S}_n$ in~(\ref{lien_Tn_Zn}), and to use~(\ref{PGD_cond}) and the lower bound on~\smash{$\Probapn\big(\mathcal{S}_n\big)$} derived in lemma~\ref{Mn_FK_lemme_minoration_Z_n}.
\end{proof}

\section{Conclusion}

We summarize here how to combine the different estimates to obtain our result.

\begin{proof}[Proof of theorem~\ref{Ising_thm_CV_avec_dep_Messikh}]
\correct{Let~$a\in(31/16,2)$ such that~$\mathcal{FSS}(16-8a)$ holds. Given that~$a<2$, the lower bound on~$Z_n$ obtained in lemma~\ref{Ising_lemme_minoration_Z_n} implies that}
$$\liminfn\,\frac{\ln Z_n}{n}\ \geqslant\ 0\,.$$
\correct{Plugging this into the computation~(\ref{outline_majo_surcritique}) and using lemma~\ref{Ising_majo_surcritique}, we get}
\begin{multline*}
\forall\varepsilon>0\qquad
\limsupn\,\frac{1}{n}\,\ln\mu_n\Big(\,T_n\geqslant T_c+\varepsilon\,\Big)\\
\ \leqslant\ \limsupn\,\frac{1}{n}\,\sup_{T\geqslant T_c+\varepsilon}\,\ln\mu_{n,\,T}^+\Big(\,\abs{m}\geqslant n^a\sqrt{T_c+\varepsilon}\,\Big)\ <\ 0\,.
\end{multline*}
\correct{The control of~\smash{$\mu_n\big(T_n\leqslant T_c-\varepsilon\big)$} is obtained similarly, using lemma~\ref{Ising_majo_souscritique}.}
\end{proof}

\bibliographystyle{alpha}
\bibliography{Ising2DSOC}

\end{document}